\documentclass[reqno,10pt,letterpaper]{amsart}

\usepackage[english]{babel}
\usepackage{amsfonts, amsmath, amssymb, amsthm, mathrsfs}

\usepackage{tikz,multirow}

\usepackage{hyperref}

\usepackage{mathtools}  
\mathtoolsset{showonlyrefs}	% add this in to only let referenced equations get a number attached to them

\newtheorem{definition}{Definition}[section]
\newtheorem{lemma}[definition]{Lemma}
\newtheorem{proposition}[definition]{Proposition}
\newtheorem{theorem}[definition]{Theorem}
\theoremstyle{remark}

\newtheorem{remark}[definition]{Remark}

\newcommand{\C}{\mathbb{C}}
\newcommand{\R}{\mathbb{R}}
\newcommand{\N}{\mathbb{N}}

\renewcommand{\Re}{\operatorname{Re}}
\renewcommand{\Im}{\operatorname{Im}}

\newcommand{\setleft}[2]{ \left\{ \left. #1 \,\right|\, #2 \right\} }
\newcommand{\setright}[2]{ \left\{ #1 \,\left|\, #2 \right.\right\} }

\newcommand{\Rono}{\R^{1,n+1}}
\newcommand{\Hn}{\mathbb{H}^{n+1}}
\newcommand{\Sn}{\mathbb{S}^n}

\newcommand{\End}{\operatorname{End}}
\newcommand{\T}{\mathrm{T}}

\newcommand\WF{\mathrm{WF}}
\newcommand\Char{\mathrm{Char}}

\newcommand{\norm}[2]{\| #1 \|_{#2}}

\newcommand{\supp}{\mathrm{supp}}
\newcommand{\vol}{\mathrm{vol}}
\newcommand{\even}{\mathrm{even}}
\newcommand{\loc}{\mathrm{loc}}

\newcommand{\grad}{\operatorname{grad}}

\newcommand{\n}{\nabla}
\newcommand\tr{\operatorname{tr}}
\newcommand{\Diff}{\mathrm{Diff}}

\newcommand\chris[2]{\Gamma_{#1}^{#2}}

\newcommand{\ve}{\varepsilon}
\newcommand\tsum{\textstyle\sum}

\newcommand{\rlnm}{\rho^{\lambda+\frac n2 - m}}
\newcommand{\drr}{\tfrac{d\rho}{\rho}}
\newcommand{\rdr}{{\rho\partial_\rho}}

\newcommand{\Sym}{\mathrm{Sym}}

\newcommand\sA{\mathscr{A}}
\newcommand\co[3]{{\{#1\to#2\}#3}}
\newcommand\ct[5]{{\{#1\to#2,#3\to#4\}#5}}
\newcommand\cthree[7]{{\{#1\to#2,#3\to#4,#5\to#6\}#7}}

\newcommand{\ip}{\operatorname{\lrcorner}}
\renewcommand{\sp}{\operatorname{\cdot}}

\newcommand{\trace}{\operatorname{\Lambda}}
\newcommand{\adjtrace}{\operatorname{L}}

\renewcommand{\d}{\operatorname{d}}
\renewcommand{\div}{\operatorname{\delta}}

\newcommand{\Riemann}{\operatorname{R}}
\newcommand\Ric{\operatorname{Ric}}

\usepackage{tensor}

\newcommand\threerow[3]{\left[\begin{array}{ccc}	#1 & #2 & #3 \end{array}\right]}
\newcommand\threevector[3]{\left[\begin{array}{c} 	#1 \\ #2 \\ #3 \end{array}\right]}
\newcommand\threematrix[9]{\left[\begin{array}{ccc}	#1&#2&#3 \\
										#4&#5&#6 \\
										#7&#8&#9 \end{array}\right]}

\usepackage{leftidx} % to put subscripts in upper left corner

\newcommand{\Xone}{X}
\newcommand{\Xtwo}{X_{cs}}
\newcommand{\Xthree}{X_{e}}
\newcommand{\Xdouble}{X^2}
\newcommand{\Mone}{M}
\newcommand{\Mthree}{M_{e}}

\newcommand{\CS}{S}

\newcommand{\gamb}{\eta}
\newcommand{\meas}{dx}

\newcommand\nM{\leftidx{^{\Mone}}{\n}}
\newcommand\nMthree{\leftidx{^{\Mthree}}{\n}}
\newcommand\nY{\leftidx{^Y}{\n}}

\newcommand{\Resl}{\operatorname{\mathcal{R}}_\lambda}

\newcommand\japbrak[1]{\langle#1\rangle}

\newcommand{\terms}{\mathrm{terms}} 			% used in calculation of commutation of divergence with laplacian
\newcommand{\scal}{\mathrm{scal}}				% used to denote symbol class which are actually just scalar 

\newcommand{\divterm}{\textrm{div}}

\newcommand\bbb{\mathrm{b}}					% used for Diff_b notation and for b-calculus stuff defined below

\newcommand\MRiemann{\leftidx{^{\Mone}}{\mathrm{R}}{}}
\newcommand\MthreeRiemann{\leftidx{^{\Mthree}}{\mathrm{R}}{}}

\newcommand{\s}{s}
\newcommand{\RplusS}{\R^+_\s}
\newcommand{\sds}{\s\partial_\s}
\newcommand{\dss}{\tfrac{d\s}{\s}}

\renewcommand{\t}{t}
\newcommand{\RplusT}{\R^+_\t}
\newcommand{\tdt}{\t\partial_\t}
\newcommand{\dtt}{\tfrac{d\t}{\t}}

\newcommand{\trETA}{\tr_\gamb}
\newcommand{\trH}{\tr_h}

\newcommand\innprod[2]{ \langle #1, #2 \rangle }
\newcommand\innprodF[2]{ \langle #1, #2 \rangle_F }
\newcommand\innprodSETA[2]{ \langle #1, #2 \rangle_{\s^{-2}\gamb} }
\newcommand\innprodTETA[2]{ \langle #1, #2 \rangle_{\t^{-2}\gamb} }
\newcommand\innprodETA[2]{ \langle #1, #2 \rangle_\gamb }
\newcommand\innprodTT[2]{ \langle #1, #2 \rangle_\t }
\newcommand\innprodSS[2]{ \langle #1, #2 \rangle_\s }
\newcommand\innprodHT[2]{ \langle #1, #2 \rangle_{h,\t} }

\newcommand{\Lsections}{L^2}
\newcommand{\LsectionsS}{L^2_\s}
\newcommand{\LsectionsT}{L^2_\t}

\newcommand\LinnprodSETA[2]{ ( #1 , #2 )_{\s^{-2}\eta}}
\newcommand\LinnprodTETA[2]{ ( #1 , #2 )_{\t^{-2}\eta}}
\newcommand\LinnprodETA[2]{ ( #1 , #2 )_\eta}
\newcommand\LinnprodSS[2]{ ( #1 , #2 )_\s}
\newcommand\LinnprodTT[2]{ ( #1 , #2 )_\t}

\newcommand\hookprod[2]{#1 \ip #2}
\newcommand\hookprodETA[2]{#1 \ip_\gamb #2}

\newcommand{\E}{\mathcal{E}}
\newcommand{\F}{\mathcal{F}}
\newcommand{\Ek}{\E^{(k)}}

\newcommand{\opkm}{\oplus_{k=0}^m\,}

\newcommand{\pistarS}{\pi^*_\s}
\newcommand{\pistarT}{\pi^*_\t}

\newcommand{\traceETA}{\trace_\gamb}
\newcommand{\traceF}{\trace_F}
\newcommand{\traceSETA}{ \Lambda_{\s^{-2}\gamb}}

\newcommand{\dETA}{\d_\gamb}
\newcommand{\divETA}{\div_\gamb}

\newcommand{\DeltaAmb}{\operatorname{\square}}

\newcommand{\QVasyUp}{\operatorname{\mathbf{Q}}}
\newcommand{\QVasyDown}{\operatorname{\mathcal{Q}}}
\newcommand{\QVasyDOWNL}{\QVasyDown_\lambda}

\newcommand{\PVasyUp}{\operatorname{\mathbf{P}}}
\newcommand{\PVasyDown}{\operatorname{\mathcal{P}}}
\newcommand{\PVasyDOWNL}{\PVasyDown_\lambda}

\newcommand{\DVasyUp}{\mathbf{D}}
\newcommand{\DVasyDown}{\mathcal{D}}

\newcommand{\ZeroOrderVasyUp}{\mathbf{G}}
\newcommand{\ZeroOrderVasyDown}{\mathcal{G}}

\newcommand{\BVasyDOWNL}{\mathcal{B}_\lambda}

\newcommand{\LowerOrderTermVasyUp}{\mathbf{A}}
\newcommand{\LowerOrderTermVasyDown}{\mathcal{A}}
\newcommand{\LowerOrderTermVasyDOWNL}{\LowerOrderTermVasyDown_\lambda}

\newcommand{\HVasyDown}{\mathcal{H}}

\newcommand\bT{\leftidx{^\bbb}{\T}}
\newcommand\bF{\leftidx{^\bbb}{\F}}

\newcommand\IF{\operatorname{I}}

\newcommand{\XVasyML}{\mathcal{X}}
\newcommand{\YVasyML}{\mathcal{Y}}

\newcommand\rct{\overline{\T}\vphantom{\T}^*}
\newcommand\overlineH{\overline{H}\vphantom{H}}

\newcommand\principalsymbol[1]{\sigma(#1)}
\newcommand\Hamiltonian[1]{H_{#1}}

\newcommand\Source{\Sigma_+}
\newcommand\Sink{\Sigma_-}
\newcommand\cSource{\Gamma_+}
\newcommand\cSink{\Gamma_-}
\newcommand\SourceSink{\Sigma_\pm}
\newcommand\cSourceSink{\Gamma_\pm}

\title[Resonances for symmetric tensors]
{Resonances for Symmetric Tensors on Asymptotically Hyperbolic Spaces}
\author{Charles Hadfield}
\address{DMA, \'Ecole Normale Sup\'erieure, 45 rue d'Ulm,
75230 Paris cedex 05, France}
\email{charles.hadfield@ens.fr}

\begin{document}

\begin{abstract}
On manifolds with an even Riemannian conformally compact Einstein metric, the resolvent of the Lichnerowicz Laplacian, acting on trace-free, divergence-free, symmetric 2-tensors is shown to have a meromorphic continuation to the complex plane, defining quantum resonances of this Laplacian. For higher rank symmetric tensors, a similar result is proven for (convex cocompact) quotients of hyperbolic space.
\end{abstract}

\maketitle

\thispagestyle{empty}

%%%%% %%%%% %%%%% %%%%% %%%%% %%%%% %%%%% %%%%% %%%%% %%%%% %%%%% %%%%%
%%%%% %%%%% %%%%% %%%%% %%%%% %%%%% %%%%% %%%%% %%%%% %%%%% %%%%% %%%%%
%%%%% %%%%% %%%%% %%%%% %%%%% %%%%% %%%%% %%%%% %%%%% %%%%% %%%%% %%%%%

\section{Introduction}

This paper studies the meromorphic extension of the resolvent of the Laplacian acting on symmetric tensors above asymptotically hyperbolic manifolds. 
The geometric setting of asymptotically hyperbolic manifolds, modelled on convex cocompact quotients of hyperbolic space, dates to work of Mazzeo and Melrose \cite{Mazzeo, mazzeo-melrose} and  of Fefferman and Graham \cite{fefferman-graham:ci}. The meromorphic extension with finite rank poles of the resolvent of the Laplacian on functions is obtained in \cite{mazzeo-melrose} excluding certain exceptional points in $\C$. Refining the definition of asymptotically hyperbolic manifolds by introducing a notion of evenness, Guillarmou \cite{g:duke} provides the meromorphic extension to all of $\C$ and shows that for such an extension, said evenness is essential, see also \cite{guillope-zworski:pb}. 
By shifting viewpoint and studying a Fredholm problem, rather than using Melrose's pseudodifferential calculus on manifolds with corners, Vasy \cite{v:ml:inventiones,v:ml:functions} is also able to recover the result of \cite{g:duke}. This technique is presented in a very accessible article of Zworski \cite{zworski:vm} in a microlocal language (non-semiclassical).
This alternative method is more appropriate when one considers vector bundles, and, for symmetric tensors, is lightly explained later in this introduction. Effectively contained in \cite{v:ml:inventiones}, the meromorphic extension is explicitly obtained in \cite{v:ml:forms} for the resolvent of the Hodge Laplacian upon restriction to coclosed forms (or excluding top forms, for closed forms). Such a restriction is natural in light of works in a conformal setting \cite{aubry-g, branson-gover}, i.e. the boundary of the asymptotic space. In fact, from the conformal geometry viewpoint, Vasy's method of placing the asympotically hyperbolic manifold in an ambient manifold equipped with a Lorentzian metric is very much in the spirit of both the tractor calculus \cite{branson-eastwood-gover} as well as the ambient metric construction \cite{fefferman-graham:am}.

We announce the theorems (with precise definitions of the objects involved left to the body of the article) and sketch their resolution. Let $\overline\Xone$ be a compact manifold with boundary $Y=\partial \overline\Xone$. That $(\Xone, g)$ is asymptotically hyperbolic means that, locally near $Y$ in $\overline\Xone$, there exists a chart $[0,\ve)_\rho\times Y$ such that on $(0,\ve)\times Y$, the metric $g$ takes the form
\[
g = \frac{d\rho^2 + h}{\rho^2}
\]
where $h$ is a family of Riemannian metrics on $Y$, depending smoothly on $\rho\in [0,\ve)$. That $g$ is even means that $h$ has a Taylor series about $\rho=0$ in which only even powers of $\rho$ appear. Above $\Xone$, we consider the set of symmetric cotensors of rank $m$, denoting this vector bundle $\E^{(m)}=\Sym^m \T^*\Xone$. On symmetric tensors, there exist two common Laplacians. The (positive) rough Laplacian $\n^*\n$ and the Lichnerowicz Laplacian $\Delta$, originally defined on 2-cotensors \cite{lichnerowicz}, but easily extendible to arbitrary degree \cite{hms}. On functions, these two Laplacians coincide, on one forms, the Lichnerowicz Laplacian agrees with the Hodge Laplacian, and in general, for symmetric $m$-cotensors, the Lichnerowicz Laplacian differs from the rough Laplacian by a zeroth order curvature operator
\[
\Delta = \n^*\n + q(\Riemann).
\] 
We construct the Lorentzian cone $\Mone=\RplusS\times\Xone$ with metric $\gamb=-d\s\otimes d\s + \s^2 g$ (and call $\s$ the Lorentzian scale). Pulling $\E^{(m)}$ back to $\Mone$ we naturally see $\E^{(m)}$ as a subbundle of the bundle of all symmetric cotensors of rank $m$ above $\Mone$, this larger bundle is denoted $\F = \Sym^m \T^*\Mone$. On $\F$ we consider the Lichnerowicz d'Alembertian $\DeltaAmb$. Up to symmetric powers of $\dss$ we may identify $\F$ with the direct sum of $\E^{(k)} = \Sym^k \T^*\Xone$ for all $k\le m$. Indeed by denoting $\E=\opkm \Ek$ the bundle of all symmetric tensors above $\Xone$ of rank not greater than $m$, we are able to pull back  sections of this bundle and see them as sections of $\F$:
\[
\pistarS : 
			C^\infty(\Xone; \E)  \to  C^\infty (\Mone; \F).
\]
A long calculation gives the structure of the Lichnerowicz d'Alembertian with respect to this identification. It is seen that $\s^2\DeltaAmb$ decomposes as the Lichnerowicz Laplacian $\Delta$ acting on each subbundle of $\Ek$ for $0\le k\le m$ however these fibres are coupled via off-diagonal terms consisting of the symmetric differential $\d$ and its adjoint, the divergence $\div$. (There are also less important couplings due to the trace $\trace$ and its adjoint $\adjtrace$.) Also present in the diagonal are terms involving $\sds$ and $(\sds)^2$. By conjugating by $\s^{-\frac n2+m}$ we obtain an operator 
\[
\QVasyUp = \n^*\n + (\sds)^2 + \DVasyUp + \ZeroOrderVasyUp
\]
where $\DVasyUp$ is of first order consisting of the symmetric differential and the divergence, while $\ZeroOrderVasyUp$ is a smooth endomorphism on $\F$. By appealing to the b-calculus of Melrose \cite{melrose:aps}, we can push this operator acting on $\F$ above $\Mone$ to a family of operators (holomorphic in the complex variable $\lambda$) acting on $\E$ above $\Xone$ of the form
\[
\QVasyDOWNL = \n^*\n + \lambda^2 + \DVasyDown + \ZeroOrderVasyDown
\]
where $\DVasyDown$ is of first order consisting of the symmetric differential and the divergence, while $\ZeroOrderVasyDown$ is a smooth endomorphism on $\E$. Explicitly, in matrix notation writing
\[
u = \threevector{u^{(m)}}{\vdots}{u^{(0)}},
\qquad
u\in C^\infty(\Xone; \E), 
u^{(k)}\in C^\infty(\Xone;\Ek)
\]
the operator $\QVasyDOWNL$ takes the following form
\newcommand\tikzmark[1]{%
  \tikz[overlay,remember picture,baseline] \node [anchor=base] (#1) {};}
\newcommand\MyLine[3][]{%
  \begin{tikzpicture}[overlay,remember picture]
    \draw[#1,loosely dotted] (#2.north west) -- (#3.south east);
  \end{tikzpicture}}
\[\setlength\unitlength{10pt}	% this line needs to be in agreement with overall size defined by document
\left[
\begin{picture}(35.5,6)(-4.5,0)
\put(0,5){\makebox(0,0){$\Delta+\lambda^2 - c_m - \adjtrace\trace$}}
																					\put(6.8,4.9){\makebox(0,0){$2b_{m-1}\d$}}
																																			\put(12,4.9){\makebox(0,0){$-b_{m-2}b_{m-1}\adjtrace$}}
\put(0,3){\makebox(0,0){$-2b_{m-1}\div$}}
\put(0,1){\makebox(0,0){$-b_{m-2}b_{m-1}\trace$}}

																																			\put(26.7,-1){\makebox(0,0){$-b_0b_1\adjtrace$}}
																																			\put(26.7,-3){\makebox(0,0){$2b_0\d$}}
\put(14.7,-5.05){\makebox(0,0){$-b_{0}b_{1}\trace$}}
																					\put(19.9,-5.05){\makebox(0,0){$-2b_0\div$}}					
																																			\put(26.7,-5){\makebox(0,0){$\Delta+\lambda^2 - c_0 - \adjtrace\trace$}}

\put(15,3){\tikzmark{upup1}}				\put(21.7,1){\tikzmark{upup2}}
\put(10,3){\tikzmark{up1}}					\put(21.7,-1){\tikzmark{up2}}
\put(5,3){\tikzmark{center1}}				\put(21.7,-3){\tikzmark{center2}}
\put(5,1){\tikzmark{down1}}				\put(16.7,-3){\tikzmark{down2}}
\put(5,-1){\tikzmark{downdown1}}			\put(11.7,-3){\tikzmark{downdown2}}

\put(23,3){\mbox{\Huge 0}}
\put(3,-4){\mbox{\Huge 0}}
\end{picture}
\right]
\]
for constants
\MyLine[thick]{upup1}{upup2}
\MyLine[thick]{up1}{up2}
\MyLine[thick]{center1}{center2}
\MyLine[thick]{down1}{down2}
\MyLine[thick]{downdown1}{downdown2}
\[
b_k = \sqrt{m-k},
\qquad
c_k = \tfrac{n^2}{4} + m(n+2k+1) - k(2n+3k-1)
\]
and operators: $\Delta$ the Lichnerowicz Laplacian; $\div$ the divergence; $\d$ the symmetric differential; $\trace$ the trace; $\adjtrace$ the adjoint of the trace. (The operator $\QVasyDOWNL$ naively does not appear self-adjoint for $\lambda\in i\R$ since $\div$ is the adjoint of $\d$. The sign discrepancy is due to the Lorentzian signature of $\gamb$. The operator is indeed self-adjoint for $\lambda\in i\R$ as detailed in Proposition~\ref{prop:QQ:selfadjoint}.) When this family of operators acts on $\Lsections$ sections, denoted $\LsectionsS(\Xone;\E)$ described in \eqref{eqn:L2definitionforEwithS}, it has an inverse for $\Re\lambda\gg1$. This family of operators has the following meromorphic family of inverses
\begin{theorem}\label{thm:QQ}
Let $(\Xone^{n+1},g)$ be even asymptotically hyperbolic. Then the inverse of (Definition~\ref{defn:QQ})
\[
\QVasyDOWNL \textrm{ acting on } \LsectionsS(\Xone;\E)
\]
written $\QVasyDOWNL^{-1}$ has a meromorphic continuation from $\Re\lambda\gg 1 $ to $\C$,
\[
\QVasyDOWNL^{-1} : C_c^\infty ( \Xone ; \E )\to \rlnm \bigoplus_{k=0}^m \rho^{-2k} C^{\infty}_\even ( \overline\Xone ; \Ek )
\]
with finite rank poles.
\end{theorem}

Consider $u\in C^\infty(\Xone; \E)$. Although the trace operator $\trace$ acting on each subbundle $\Ek$ gives a notion of $u$ being trace-free, it is more natural to consider the ambient trace operator from $\F$, denoted $\traceETA$ (Subsection~\ref{subsec:vecbundles}). Pulling $u$ back to $\Mone$, we have $\pistarS u\in C^\infty (\Mone; \F)$ and we may consider the condition that $\pistarS u \in \ker \traceETA$. Avoiding extra notation for this subbundle of $\E$ (consisting of symmetric tensors above $\Xone$ which are trace-free with respect to the ambient trace operator $\traceETA$) we will simply refer to its sections using the notation
\[
C^\infty(\Xone;\E) \cap \ker ( \traceETA \circ \pistarS)
\]
On this subbundle, the operator $\QVasyDOWNL$ takes the following form
\[\setlength\unitlength{10pt}	% this line needs to be in agreement with overall size defined by document
\left[
\begin{picture}(25.8,5.5)(-3,0)
\put(0,4.3){\makebox(0,0){$\Delta+\lambda^2 - c_m'$}}
																					\put(6.8,4.2){\makebox(0,0){$2b_{m-1}\d$}}
																																			%\put(12,4.9){\makebox(0,0){$-b_{m-1}b_{m-2}\L$}}
\put(0,2){\makebox(0,0){$-2b_{m-1}\div$}}
%\put(0,1){\makebox(0,0){$-b_{m-2}b_{m-1}\,\Lambda$}}

																																			%\put(26.7,-1){\makebox(0,0){$-b_0b_1\L$}}
																																			\put(19.9,-2.1){\makebox(0,0){$2b_0\d$}}
%\put(14.7,-5.05){\makebox(0,0){$-b_{0}b_{1}\,\Lambda$}}
																					\put(13.1,-4.35){\makebox(0,0){$-2b_0\div$}}					
																																			\put(19.9,-4.3){\makebox(0,0){$\Delta+\lambda^2 - c_0'$}}

\put(8.5,2.5){\tikzmark{up1}}					\put(16,-0.5){\tikzmark{up2}}
\put(4,2.5){\tikzmark{center1}}				\put(16,-2.5){\tikzmark{center2}}
\put(4,0.5){\tikzmark{down1}}				\put(11.5,-2.5){\tikzmark{down2}}

\put(17,2){\mbox{\Huge 0}}
\put(2,-3){\mbox{\Huge 0}}
\end{picture}
\right]
\]
with the modified constants
\MyLine[thick]{up1}{up2}
\MyLine[thick]{center1}{center2}
\MyLine[thick]{down1}{down2}
\[
c_k'=c_k - (m-k)(m-k-1).
\]
Note that if $u=u^{(m)}\in C^\infty(\Xone;\E^{(m)})$ then $u\in \ker\trace$ if and only if $\pistarS u \in \ker \traceETA$. Again, a similar meromorphic extension of the inverse may be obtained.
\begin{theorem}\label{thm:QQ:tracefree}
Let $(\Xone^{n+1},g)$ be even asymptotically hyperbolic. Then the inverse of (Definition~\ref{defn:QQ})
\[
\QVasyDOWNL \textrm{ acting on } \LsectionsS(\Xone;\E) \cap \ker ( \traceETA \circ \pistarS)
\]
written $\QVasyDOWNL^{-1}$ has a meromorphic continuation from $\Re\lambda\gg 1 $ to $\C$,
\[
\QVasyDOWNL^{-1} : C_c^\infty ( \Xone ; \E ) \cap \ker ( \traceETA \circ \pistarS) \to \rlnm \left( \bigoplus_{k=0}^m \rho^{-2k} C^{\infty}_\even ( \overline\Xone ; \Ek ) \right) \cap \ker ( \traceETA \circ \pistarS)
\]
with finite rank poles.
\end{theorem}

In order to uncouple the Lichnerowicz Laplacian acting on $\E^{(m)}$ and obtain the desired meromorphic extension of the resolvent, we need to restrict further from simply trace-free tensors to trace-free divergence-free tensors. Equivalently, we must be able to commute the Lichnerowicz Laplacian with both the trace operator and the divergence operator. The first commutation is always possible giving the preceding structure of $\QVasyDOWNL$ however, unlike in the setting of differential forms (where the Hodge Laplacian always commutes with the divergence), such a commutation on symmetric tensors depends on the geometry of $(\Xone,g)$. For $m=2$ the condition is that the Ricci tensor be parallel, while for $m\ge 3$, the manifold must be locally isomorphic to hyperbolic space.
\begin{theorem}\label{thm:main:2}
Let $(\Xone^{n+1},g)$ be even asymptotically hyperbolic and Einstein. Then the inverse of
\[
\Delta - \frac{n(n-8)}{4}+\lambda^2 \textrm{ acting on } \Lsections( \Xone ; \E^{(2)} )\cap\ker\trace\cap\ker\div
\]
written $\Resl$ has a meromorphic continuation from $\Re\lambda\gg 1 $ to $\C$,
\[
\Resl : C_c^\infty ( \Xone ; \E^{(2)} )\cap\ker\trace\cap\ker\div\to \rho^{\lambda+\frac n2-2} C^{\infty}_\even ( \overline\Xone ; \E^{(2)} )\cap\ker\trace\cap\ker\div
\]
with finite rank poles.
\end{theorem}

\begin{theorem}\label{thm:main:m}
Let $(\Xone^{n+1},g)$ be a convex cocompact quotient of $\Hn$. Then the inverse of
\[
\Delta - \frac{n^2 - 4m(n+m-2)}{4} +\lambda^2 \textrm{ acting on } \Lsections( \Xone ; \E^{(m)} )\cap\ker\trace\cap\ker\div
\]
written $\Resl$ has a meromorphic continuation from $\Re\lambda\gg 1 $ to $\C$,
\[
\Resl : C_c^\infty ( \Xone ; \E^{(m)} )\cap\ker\trace\cap\ker\div\to \rlnm C^{\infty}_\even ( \overline\Xone ; \E^{(m)} )\cap\ker\trace\cap\ker\div
\]
with finite rank poles.
\end{theorem}
Note that on $\Hn$, the difference between the Lichnerowicz Laplacian and the rough Laplacian is $q(\Riemann) = -m(n+m-1)$. Thus by introducing a spectral parameter $s = \lambda + \tfrac n2$ (not to be confused with the Lorentzian scale), the previous operator $\Delta - c_m + \lambda^2$ may be equivalently written
\[
\n^*\n - s(n-s) - m
\] 
in the spirit of \cite{dfg}.

In order to demonstrate Theorem~\ref{thm:QQ}, Vasy's technique is to consider a slightly larger manifold $\Xthree$ as well as the ambient space $\Mthree=\R^+\times \Xthree$. Using two key tricks near the boundary $Y=\partial\overline\Xone$: the evenness property allows us to introduce the coordinate $\mu=\rho^2$ and twisting the Lorentzian scale with the boundary defining function gives (what is termed the Euclidean scale) $\t=\s/\rho$, it is seen that the ambient metric $\eta$ may be extended non-degenerately past $\R^+\times Y$ to $\Mthree$. On $\Sym^m \T^*\Mthree$ we construct analogous to $\QVasyUp$, an operator $\PVasyUp$ replacing appearances of $\s$ by $\t$ which, on $\Mone$ is easily related to $\QVasyUp$. Again the b-calculus provides a family of operators $\PVasyDown$ on $\opkm \Sym^k \T^*\Xthree$ above $\Xthree$. Section~\ref{sec:analysis} shows precisely how this family of operators fits into a Fredholm framework giving a meromorphic inverse, and very quickly also provides Theorem~\ref{thm:QQ}.

Such theorems are desirable for several reasons. Firstly, the quantum/classical correspondence between the spectrum of the Laplacian on a closed hyperbolic surface and Ruelle resonances of the generator of the geodesic flow on the unit tangent bundle \cite[Proposition 4.1]{faure-tsujii} has been extended to compact hyperbolic manifolds of arbitrary dimension \cite{dfg} at which point the correspondence is between Ruelle resonances and the spectrum of the Laplacian acting on trace-free, divergence-free, symmetric tensors of arbitrary rank. This correspondence is extended in \cite{ghw} to convex cocompact hyperbolic surfaces using the scattering operator \cite{graham-zworski} as well as \cite{dyatlov-g} to obtain Ruelle resonances in this open system. Theorem~\ref{thm:main:m} has been applied, along with results from \cite{dfg, dyatlov-g}, in order to provide such a correspondence in the setting of convex cocompact hyperbolic manifolds of arbitrary dimension \cite{h:qr}. Secondly, with knowledge of the asymptotics of the resolvent of the Laplacian on functions, it is possible to construct the Poisson operator, the Scattering operator, and study in a conformal setting, the GJMS operators and the $Q$-curvature of Branson \cite[Chapters 5,6]{djadli-g-herzlich}. This problem should be particularly interesting on symmetric 2-cotensors above a conformal manifold which, upon extension to a ``bulk" Poincar\'e-Einstein manifold, makes contact with Theorem~\ref{thm:main:2}. Finally, and again with respect to Theorem~\ref{thm:main:2}, the Lichnerowicz Laplacian plays a fundamental role in problems involving deformations of metrics and their Ricci tensors \cite{biquard, delay:el, graham-lee} as well as to linearised gravity \cite{wang}. Spectral analysis of the Lichnerowicz Laplacian \cite{delay:es,delay:tt} as well as the desire to build a scattering operator emphasise the importance of considering this Laplacian acting on more general spaces than that of $\Lsections$ sections. From the viewpoint of gravitational waves, the recent work \cite{baskin-v-wunsch} studies decay rates of solutions to the wave equation (acting on the trivial bundle) on Minkowski space with metrics similar to \eqref{eqn:rho-2eta}. It is very natural to consider this problem on symmetric 2-cotensors acted upon by the Lichnerowicz d'Alembertian.

Theorem~\ref{thm:main:2} requires the global condition that the manifold be Einstein. It is unclear whether such a condition is necessary. Vasy's technique deals with the condition of even asymptotic hyperbolicity near the boundary. Indeed, this is reflected in Theorem~\ref{thm:QQ:tracefree}. However to obtain our desired result, uncoupling the Lichnerowicz Laplacian from the operator $\QVasyDown$ currently requires a global condition on the base manifold. One should study whether perturbation techniques could provide a more general theorem giving precise conditions for when such a meromorphic continuation exists.

The paper is structured as follows. Section~\ref{sec:geom} sets up the geometric side of the problem introducing the various manifolds of interest as well as the construction of the ambient metric $\gamb$. This section also includes a digression into the model geometry $\Xone=\Hn$ to motivate Vasy's construction. Section~\ref{sec:symtensors} introduces the algebraic aspects of symmetric tensors, introduces many notational conventions and establishes several relationships between symmetric tensors when working relative to the Lorentzian and Euclidean scales. Section~\ref{sec:bcalc:microlocal} recalls standard notions from microlocal analysis and explicits several notions from the b-calculus framework adapted to vector bundles. Section~\ref{sec:Lap.DAlem.Q} contains the bulk of the calculations of this paper, relating $\QVasyUp$ and $\QVasyDown$ with the Lichnerowicz Laplacian. 
%It also explicits the case $m=2$ in the hope that such an exposition of this low rank case aids intuition. 
Sections~\ref{sec:P} and~\ref{sec:analysis} introduce the operators $\PVasyUp$ and $\PVasyDown$ and provide the desired meromorphic inverse. Section~\ref{sec:proofs:main} establishes the four theorems.
Section~\ref{sec:m=2} details the particular case of symmetric cotensors of rank $m=2$. It is useful to gain insight into this problem via this low rank setting, and it is hoped that the presentation of this case will aid the reader particularly during Sections~\ref{sec:Lap.DAlem.Q} and~\ref{sec:proofs:main}.
Finally, Section~\ref{sec:highenergyestimates} announces the high energy estimates one would obtain if the microlocal analysis performed in Section~\ref{sec:analysis} was performed using semiclassical notions.

\subsection*{Acknowledgements} I would like to thank Colin Guillarmou for his guidance and for the many discussions placing this work in a wider context. Also Andrei Moroianu for his continued support and several useful comments on Section~\ref{sec:Lap.DAlem.Q}. Finally Rod Gover for originally sparking my interest in all things conformal.

%%%%% %%%%% %%%%% %%%%% %%%%% %%%%% %%%%% %%%%% %%%%% %%%%% %%%%% %%%%%
%%%%% %%%%% %%%%% %%%%% %%%%% %%%%% %%%%% %%%%% %%%%% %%%%% %%%%% %%%%%
%%%%% %%%%% %%%%% %%%%% %%%%% %%%%% %%%%% %%%%% %%%%% %%%%% %%%%% %%%%%

\section{Geometry}\label{sec:geom}

\subsection{Model geometry}

It is worth mentioning the model geometry which provides a clear geometric motivation for the construction of the ambient space as well as the Minkowski and Euclidean scales.

Let 
$\Rono$ 
be Minkowski space with the Lorentzian metric
\begin{align}
\gamb 
:= 
-dx_0^2 + \sum_{i=1}^{n+1} dx_i^2
\end{align}
and set 
$\Mthree$ 
to be Minkowski space minus the closure of the backward light cone. The metric gives the Minkowski distance function, denoted 
$\gamb^2$, 
on 
$\Rono$ 
from the origin:
\begin{align}
\gamb^2(x)
:= 
- x_0^2 + \sum_{i=1}^{n+1} x_i^2.
\end{align}
Hyperbolic space 
$\Xone=\Hn$ 
is then identified with the (connected) hypersurface 
\begin{align}
\Xone
:=
\setleft{x\in\Rono}{\gamb^2(x)=-1, x_{0}>0}
\end{align}
and is given the metric 
$g$ 
induced by the restriction of 
$\gamb$. 
The boundary at infinity of hyperbolic space, i.e. the sphere 
$Y=\Sn$, 
is identified with the (connected) submanifold
\begin{align}
Y
:=
\setleft{x\in\Rono}{\gamb^2(x)=0, x_{0}=1}
\end{align}
which, as an aside, inherits the standard metric, denoted 
$h$,
by restriction of 
$\gamb$. 
For completeness we introduce de Sitter space 
$dS^{n+1}$ 
as the hypersurface 
\begin{align}
dS^{n+1}
:=
\setleft{x\in\Rono}{\gamb^2(x)=1}.
\end{align}
We define the forward light cone 
\begin{align}
\Mone
:= 
\setleft{x \in \Rono}{\gamb^2(x)<0, x_{0}>0}
\end{align}
and note the decomposition 
$\Mone=\RplusS\times \Xone$ 
via the identification
\begin{align}
\RplusS \times \Xone \ni (\s, x) \mapsto \s\cdot x \in \Xone.
\end{align}
In these coordinates, the metric 
$\gamb$ 
restricted to 
$\Mone$ 
takes the form
\begin{align}
\gamb=-d\s \otimes d\s + \s^{2}g
\end{align}
and we refer to 
$\s$ 
as the Minkowski scale. We define 
$\Xthree$ 
to be the subset of the 
$(n+1)$-sphere 
contained in 
$\Mthree$
\begin{align}
\Xthree 
:= 
\setright{x\in\Rono}{\sum_{i=0}^{n+1} x_i^2 =1, x_{0} > \tfrac{-1}{\sqrt2}}
\end{align}
and note that the ambient space 
$\Mthree$ 
is diffeomorphic to 
$\RplusT\times \Xthree$ 
via the identification
\begin{align}
\RplusT \times \Xthree \ni (\t, x) \mapsto \t\cdot x \in \Mthree.
\end{align}
We refer to 
$\t$ 
as the Euclidean scale. The dilations induced by the Euclidean scale allow the following identification
\begin{align}
\Xthree \simeq \Xone \sqcup Y \sqcup dS^{n+1}.
\end{align}

\subsection{General setting}

We now properly introduce the geometric setting of the article. Let 
$(\Xone,g)$ 
be a Riemannian manifold of dimension 
$n+1$ 
which is even asymptotically hyperbolic 
\cite[Definition 1.2]{g:duke}
with boundary at infinity denoted 
$Y$. 
We recall the definition of evenness.
\begin{definition}
Let 
$(\Xone,g)$ 
be an asymptotically hyperbolic manifold. 
We say that 
$g$ 
is even if there exists a boundary defining function 
$\rho$ 
and a family of tensors 
$(h_{2i})_{i\in\N_0}$ 
on 
$Y=\partial \overline\Xone$ 
such that, for all 
$N$, 
one has the following decomposition of 
$g$ 
near 
$Y$
\begin{align}
\phi^*(\rho^2 g) = dr^2 + \sum_{i=0}^N h_{2i}r^{2i} + O(r^{2N+2})
\end{align}
where 
$\phi$ 
is the diffeomorphism induced by the flow 
$\phi_r$ 
of the gradient 
$\grad_{\rho^2g}(\rho)$:
\begin{align}
\phi : \left\{	 \begin{array}{rcl}
			[0,1)\times Y & \to & \phi([0,1)\times Y)\subset \overline\Xone \\
			(r,y) & \mapsto & \phi_r(y)
		\end{array}
		\right.
\end{align}
\end{definition}

We define 
$\Xdouble := (\overline\Xone\sqcup\overline\Xone)/Y$ 
to be the topological double of 
$\overline\Xone$. 
(For a slicker definition, we stray ever so slightly from the model geometry.) 
From the diffeomorphism 
$\phi$ 
we initially construct a 
$C^\infty$ 
atlas on 
$\Xdouble$ 
by noting that 
$Y\subset \Xdouble$ 
is contained in an open set 
$U^2:=(U_-\sqcup U_+)/ Y$ 
with 
$U_\pm:=\phi([0,1)\times Y)$ 
and we declare this set to be 
$C^\infty$ 
diffeomorphic to 
$(-1,1)\times Y$ 
via
\begin{eqnarray*}
(-1,1)\times Y &\simeq& U^2 \\
(t,y)&\mapsto& \left\{	\begin{array}{ll}
				\phi_{-t}(y)\in U_-, & \textrm{if }t\le 0 \\
				\phi_{+t}(y)\in U_+, & \textrm{if } t\ge 0
			\end{array}
		\right.
\end{eqnarray*}
Charts on the interior of $\Xone$ in $\overline\Xone$ complete the atlas on $\Xdouble$.

We want to consider the boundary defining function 
$\rho$ 
as a function from 
$\Xdouble$ 
to 
$[-1,1]$ 
such that 
$\Xone$ 
may be identified with 
$\{\rho>0\}$. 
Using the previous chart for 
$U^2\simeq (-1,1)\times Y$ 
we initially set 
\begin{align}
\rho : \left\{	 \begin{array}{rcl}
			(-1,1)\times Y & \to & (-1,1) \\
			(r,y) & \mapsto & r
		\end{array}
		\right.
\end{align}
and extend 
$\rho$ 
to a continuous function on 
$\Xdouble$ 
by demanding that 
$\rho$ 
be constant on 
$\Xdouble\backslash U^2$. 
In order to ensure smoothness at 
$\partial \overline{U^2}$ 
we deform 
$\rho$ 
smoothly on the two subsets 
$(-1,-1+\ve)\times Y$ 
and 
$(1-\ve,1)\times Y$ 
of $U^2$. 
This achieves our goal. We now define the function 
$\mu$ 
on 
$\Xdouble$ 
by declaring
\begin{align}
\mu : \Xdouble\to [-1,1] :
		\left\{	\begin{array}{ll}
				\mu=-\rho^2, & \textrm{if }\rho\le 0 \\
				\mu=\rho^2 & \textrm{if } \rho\ge 0
			\end{array}
		\right.
\end{align}
\begin{remark}
Although we have performed a deformation of 
$\rho$ 
near 
$\partial \overline{U^2}$ 
we will continue to think of 
$\rho$ 
and 
$\mu$ 
as coordinates for the first factor of 
$U^2=(-1,1)\times Y$ 
(if we wanted to be correct, in what follows we would replace 
$(-1,1)$ 
with 
$(-1+\ve,1-\ve)$ 
but this is cumbersome and we prefer to free up the variable 
$\ve$). 
Of course, only the coordinates 
$(\mu,y)$ 
provide a smooth chart for 
$\Xdouble$ 
near 
$Y$.
\end{remark}
We now weaken the atlas on 
$\Xdouble$ 
near 
$Y$. 
By the previous remark, we may think of 
$\mu$ 
as coordinates for the first factor of 
$U^2$ 
and we thus demand that the 
$C^\infty$ 
atlas is with respect to this coordinate rather than 
$\rho$ 
(as was the case for the initial atlas). It is now the case that on 
$\Xdouble$, 
only 
$\mu$ 
(and not $\rho$) 
is a smooth function.

We define the set 
$C^\infty_\even(\overline\Xone)$ 
to be the subset of functions in 
$C^\infty(\Xone)$ 
which are extensible to 
$C^\infty(\Xdouble)$ 
and whose extension is invariant with respect to the natural involution on 
$\Xdouble$. 
(For example, the restriction of 
$\mu$ 
to 
$\Xone$. 
However such an invariant extension would of course not give the function 
$\mu$ 
previously constructed due to a sign discrepency.) We remark that 
$\dot C^\infty(\Xone)$,
the subset of functions in 
$C^\infty(\overline\Xone)$ 
which vanish to all orders at 
$Y$, 
injects naturally into 
$C^\infty(\Xdouble)$ 
and may be identified with the subset of 
$C^\infty (\Xdouble)$ 
whose elements vanish on 
$\{\rho<0\}$. 
Such constructions may also readily be extended to the setting of vector bundles above 
$\Xone$
by using a local basis near
$Y$
of such a vector bundle which smoothly extends across
$Y$.

\begin{definition}\label{def:threebasemanifolds}
We denote by 
$\Xthree$ 
the following extension of 
$\Xone$
\begin{align}
\Xthree
:=
\{\mu>-1\}\subset \Xdouble,
\end{align}
by 
$\CS$ 
the hypersurface 
$\{\mu=-\tfrac 12\}\subset\Xthree$, 
and by 
$\Xtwo$ 
the open submanifold 
$\{\mu> -\tfrac 12\}\subset\Xthree$ 
such that 
$\partial \overline\Xtwo=\CS$.
\end{definition}

We construct two product manifolds 
$\Mone:=\RplusS \times \Xone$ 
and 
$\Mthree:=\RplusT\times \Xthree$. 
We supply 
$\Mone$ 
with the Lorentzian cone metric
\begin{align}\label{eqn:etaLcone}
\gamb := -d\s\otimes d\s + \s^{2}g
\end{align}
and explain how this structure may be smoothly extended to $\Mthree$. 

Using the even neighbourhod at infinity 
$U:=(0,1)_\mu\times Y$ 
we remark that, on 
$\RplusS\times U$, 
the Lorentzian metric takes the form
\begin{align}\label{eqn:etaLcone:even}
\gamb = - d\s\otimes d\s + \s^{2}\left( \frac{d\mu\otimes d\mu}{4\mu^2} + \frac{h}{\mu} \right)
\end{align}
where 
$h$ 
has a smooth Taylor expansion about 
$\mu=0$ 
by the evenness hypothesis. Upon the change of variables 
$\t=\s/\rho$ 
with 
$\t\in\R^+$, 
the metric, on 
$\RplusT\times U$ 
takes the form
\begin{align}
\gamb = -\mu d\t\otimes d\t - \tfrac12 \t(d\mu\otimes d\t + d\t\otimes d\mu) + \t^2 h
\end{align}
or, in a slightly more attractive convention,
\begin{align}\label{eqn:rho-2eta}
\t^{-2} \gamb = -\tfrac\mu2 (\dtt)^2 - \tfrac12 \dtt\sp d\mu + h
\end{align}
with the convention for the symmetric product $\cdot$ introduced in the following section.
From this display we see that, by extending 
$h$ 
to a family of Riemannian metrics on 
$Y$ 
parametrised smoothly by 
$\mu\in(-1,1)$, 
we can extend 
$\gamb$ 
smoothly onto the chart 
$\RplusT\times U^2 \subset \Mthree$. 
We do this, thus furnishing 
$\Mthree$ 
with a Lorentzian metric. As in the model geometry we refer to 
$\s$ 
(which is only defined on 
$\Mone$) 
as the Minkowski scale, and to 
$\t$ 
(which is defined on 
$\Mthree$) as the Euclidean scale.

From 
\eqref{eqn:rho-2eta}, 
the measure associated with 
$\t^{-2}\gamb$ 
on 
$\RplusT\times U^2$ 
is 
$\dtt \meas$ 
where 
$\meas=\frac12 d\mu\,d\vol_h$. 
On 
$U$, 
we have 
$\meas = \rho^{n+2} d\vol_g$, 
hence 
$\meas$ 
extends smoothly to a measure on 
$\Xthree$, 
also denoted 
$\meas$, 
and which agrees with 
$d\vol_g$ 
on 
$\Xone\backslash U$.

%%%%% %%%%% %%%%% %%%%% %%%%% %%%%% %%%%% %%%%% %%%%% %%%%% %%%%% %%%%%
%%%%% %%%%% %%%%% %%%%% %%%%% %%%%% %%%%% %%%%% %%%%% %%%%% %%%%% %%%%%
%%%%% %%%%% %%%%% %%%%% %%%%% %%%%% %%%%% %%%%% %%%%% %%%%% %%%%% %%%%%

\section{Symmetric Tensors}\label{sec:symtensors}

This section introduces the necessary algebraic aspects of symmetric tensors and establishes conventions (which follow \cite{hms}).

\subsection{A single fibre}\label{subsec:singlefibre}

Let $E$ be a vector space of dimension $n+1$ equipped with an inner product $g$ and let $\{e_i\}_{i=0}^n$ be an orthonormal basis and $\{e^i\}_{i=0}^n$ the corresponding dual basis for $E^*$. We denote by $\Sym^kE^*$ the $k$-fold symmetric tensor product of $E^*$. Elements are symmetrised tensor products
\[
u_1\sp \ldots \sp u_k := \sum_{\sigma \in \Pi_k} u_{\sigma(1)}\otimes \ldots\otimes u_{\sigma(k)},
\qquad
u_i\in E^*
\]
where $\Pi_k$ is the permutation group of $\{1,\dots,k\}$. By linearity, this extends the operation $\sp$ to a map from $\Sym^kE^*\times \Sym^{k'}E^*$ to $\Sym^{k+k'}E^*$. Note the inner product takes the form $g=\tfrac12\sum_{i=0}^n e^i\sp e^i$ and that for $u\in E^*$ we write $u^k$ to denote the symmetric product of $k$ copies of $u$. The inner product induces an inner product on $\Sym^kE^*$ defined by
\[
\langle u_1\sp\ldots\sp u_k, v_1\sp\ldots\sp v_k \rangle := \sum_{\sigma\in\Pi_k} g^{-1}(u_1,v_{\sigma(1)})\ldots g^{-1}(u_k,v_{\sigma(k)}),
\qquad 
u_{i},v_{i}\in E^*.
\]
For $u\in E^*$, the metric adjoint of the linear map $u\sp : \Sym^kE^*\to\Sym^{k+1}E^*$ is the contraction $\hookprod{u}{} : \Sym^{k+1}E^*\to\Sym^kE^*$ defined by
\[
(\hookprod{u}{v}) (w_1,\dots,w_k) := v(u^\#,w_1,\dots,w_k),
\qquad
u\in E^*, 
v\in \Sym^k E^*,
w_i\in E
\]
where $u^\#$ is dual to $u$ relative to the inner product on $E$. Contraction and multiplication with the metric $g$ define two additional linear maps:
\[
\trace : \left\{	 \begin{array}{rcl}
			\Sym^k E^* & \to & \Sym^{k-2} E^* \\
			u & \mapsto & \sum_{i=0}^n   \hookprod{e^i}{\hookprod{e^i}{u}}
		\end{array}
		\right.
\]
and
\[
\adjtrace : \left\{	 \begin{array}{rcl}
			\Sym^k E^* & \to & \Sym^{k+2} E^* \\
			u & \mapsto & \sum_{i=0}^n e^i\sp e^i \sp u
		\end{array}
		\right.
\]
which are adjoint to each other. As the notation is motivated by standard notation from complex geometry, we will refer to these two operators as Lefschetz-type operators.

Let $F$ be the vector space $\R\times E$ equipped with the standard Lorentzian inner product $-f\otimes f + g$ where $f$ is the canonical vector in $\R^*$. The previous constructions have obvious counterparts on $F$ which will not be detailed. (For this subsection, we write $\innprodF{\cdot}{\cdot}$ for the Lorentzian inner product on $\Sym^m F^*$.) The decomposition of $F$ provides a decomposition of $\Sym^mF^*$:
\[
\Sym^mF^* = \bigoplus_{k=0}^m  a_{k} \, f^{m-k} \sp \Sym^k E^*,
\qquad
a_k = \tfrac{1}{\sqrt{(m-k)!}}
\]
and we write
\[
u = \sum_{k=0}^m a_k \, f^{m-k} \sp u^{(k)},
\qquad
u\in \Sym^mF^*, 
u^{(k)}\in \Sym^kE^*.
\]
The choice of the normalising constant $a_k$ is chosen so that $\innprodF{u}{v} = \sum_{k=0}^m (-1)^{m-k} \innprod{u^{(k)}}{v^{(k)}}$.
There is a simple relationship between the terms $u^{(k)}$ in this decomposition of $u$ when $u$ is trace-free.
\begin{lemma}\label{lem:trace:FE}
Let $\traceF$ and $\trace$ denote the Lefschetz-type trace operators obtained from the inner products on $F$ and $E$ respectively. For $u\in \Sym^m F^*$ in the kernel of $\traceF$, we have
\[
\trace u^{(k)} = - b_{k-2}b_{k-1} u^{(k-2)}
\]
where $u=\sum_{k=0}^m a_k f^{m-k}\sp u^{(k)}$ for $u^{(k)}\in\Sym^k E^*$ and constants $b_k = \sqrt{m-k}$.
\end{lemma}
\begin{proof}
Beginning with $\traceF f^{m-k} = (m-k)(m-k-1)f^{m-k-2}$ we obtain
\[
\traceF\left( a_k \, f^{m-k} \sp u^{(k)} \right)
 =
a_{k+2} \sqrt{(m-k)(m-k-1)} f^{m-k-2} \sp  u^{(k)} + a_k \, f^{m-k} \sp \trace u^{(k)}.
\]
Therefore, as $u\in\ker\traceF$, equating powers of $f$ in the resulting formula for
\[
\traceF\left( \sum_{k=0}^m a_k \, f^{m-k} \sp u^{(k)} \right)
\]
gives
\[
a_k f^{m-k} \trace u^{(k)} + a_{k} \sqrt{(m-k+2)(m-k+1)} f^{m-k} u^{(k-2)}=0. \qedhere
\]
\end{proof}

We introduce some notation for finite sequences to simplify the calculations below. Denote by $\sA^k$ the space of all sequences $K=k_1\dots k_k$ with $0\leq k_r\leq n$. We write $\co{k_r}{j}{K}$ for the result of replacing the $r$\textsuperscript{th} element of $K$ by $j$. If $j$ is not present, this implies we remove the $r$\textsuperscript{th} element from $K$, while if $k_r$ is not present, this implies we add $j$ to $K$ to obtain $jK$. This notation extends to replacing multiple indices at once. For example, $\ct{k_p}{}{k_r}{}{K}$ indicates we first remove the $r$\textsuperscript{th} element from $K$ and then remove the $p$\textsuperscript{th} element from $\co{k_r}{}{K}$. We set
\[
e^K = e^{k_1}\sp \ldots \sp e^{k_m} \in\otimes^k E^*,
\qquad
K=k_1\dots k_m \in \sA^k.
\]

\subsection{Vector bundles}\label{subsec:vecbundles}
These constructions are naturally extended to vector bundles above manifolds. 
We include this subsection in order to announce our notations and conventions. 
Consider $\Mone$ and $\Xone$ (with similar constructions for $\Mthree$ and $\Xthree$). 
We denote 
\[
\F := \Sym^m \T^*\Mone,
\qquad
\Ek := \Sym^k \T^*\Xone,
\qquad
\E := \opkm\Ek.
\]
If we want to make precise that $\F$ consists of rank $m$ symmetric cotensors, we will write $\F^{(m)}$. 
The Minkowski scale gives the decomposition $\Mone=\RplusS\times \Xone$ and we denote by $\pi$ the projection onto the second factor $\pi:\Mone\to\Xone$. (Remark that on $\Mone$ this gives the same map as the projection $\pi:\Mthree\to\Xthree$ using the Euclidean scale $\Mthree=\RplusT\times \Xthree$.) This enables $\Ek$ to be pulled back to a bundle over $\Mone$ which we will also denote by $\Ek$.

Given $u\in C^\infty(\Mone ; \F)$, we decompose $u$ in the following way
\begin{align}\label{eqn:decomp:mink}
u  = \sum_{k=0}^m a_k \, (\dss)^{m-k} \sp u^{(k)},
\qquad
u^{(k)}\in C^\infty(\Mone; \Ek)
\end{align}
where $a_k$ is the previously introduced constant $((m-k)!)^{-1/2}$. We say that such a decomposition is relative to the Minkowski scale. 

For a fixed value of $\s$, say $\s_0$, there is an identification of the corresponding subset of $\Mone $ with $\Xone$ via the map $\pi_{|\s=\s_0}$. We will thus reuse $\pi$ for the following map
\[
\pi_{\s=\s_0} : \left\{	 \begin{array}{rcl}
			C^\infty(\Mone; \F) & \to & C^\infty (\Xone; \E)\\
			u=\sum_{k=0}^m a_k \, (\dss)^{m-k} \sp u^{(k)} & \mapsto & \sum_{k=0}^m \pi_{|\s=\s_0} {u^{(k)}}
		\end{array}
		\right.
\]
and in order to map from $C^\infty(\Xone;\E)$ to $C^\infty(\Mone;\F)$, taking into account the Minkowski scale, we introduce
\[
\pistarS : \left\{	 \begin{array}{rcl}
			C^\infty(\Xone; \E) & \to & C^\infty (\Mone; \F)\\
			u=\sum_{k=0}^m u^{(k)} & \mapsto & \sum_{k=0}^m a_k \, (\dss)^{m-k} \sp \pi^* u^{(k)}
		\end{array}
		\right.
\]

On $\Mone$ we have two useful metrics. First, $\s^{-2}\eta$ which takes the model form of the metric on $F$ introduced in the previous subsection
\[
\s^{-2}\eta = -\tfrac12 (\dss)^2 + g.
\]
Second, we have the metric $\eta$ which is geometrically advantageous as it gives the Lorentzian cone metric on $\Mone$. Notationally we will distinguish the two constructions by decorating the Lefschetz-type operators with a subscript of the particular metric used. A similar decoration will be used for the two inner products on $\F$. There are two useful relationships. First,
\begin{align}\label{eqn:twoLambdas:eta}
\traceSETA u = \s^4 \traceETA u,
\qquad
u\in \F
\end{align}
and second,
\begin{align}\label{eqn:twoinnprods:eta}
\innprodSETA{u}{v} = \s^{2m} \innprodETA{u}{v},
\qquad
u,v\in \F
\end{align}
On $\Xone$, when the metric $g$ is used, no such decoration will be added. We can however make use of the metric $\s^{-2}\eta$ by appealing to $\pistarS$. We introduce $\innprodSS{\cdot}{\cdot}$ on $C^\infty(\Xone;\E)$ by declaring
\[
\innprodSS{u}{v} := \innprodSETA{\pistarS u}{\pistarS v},
\qquad
u,v\in C^\infty(\Xone;\E).
\]
Note that such a definition does not depend on the value of $\s\in\R^+$ at which point the inner product on $\F$ is applied. With this inner product given, and the measure $d\vol_g$ previously introduced, we obtain the notion of $\Lsections$ sections and define
\begin{align}\label{eqn:L2definitionforEwithS}
\LsectionsS(\Xone;\E) := \Lsections(\Xone, d\vol_g ; \E , \innprodSS{\cdot}{\cdot} )
\end{align}
whose inner product is provided by
\[
\LinnprodSS{u}{v} := \int_{\Xone} \innprodSS{ u }{ v } \, d\vol_g,
\qquad
u,v\in C^\infty_c(\Xone ; \E).
\]
On $\Xthree$, we define $\Lsections$ sections with respect to the measure $\meas$,
\[
\LsectionsT(\Xthree;\E) := \Lsections(\Xthree, \meas ; \E , \innprodTT{\cdot}{\cdot} ).
\]
On $\Xone$, the necessary correspondences between the constructions using the Lorentzian and Euclidean scales are given in the following lemma.
\begin{lemma}\label{lem:innprod:changeofscales}
There exists $J\in C^\infty(\Xone;\End\E)$ such that
\[
\pistarS u = \pistarT Ju,
\qquad
u\in C^\infty(\Xone;\E)
\]
whose entries are homogeneous polynomials of degree at most $m$ in $\drr$, upper triangular in the sense that $J ( \E^{(k_0)} ) \subset \oplus_{k=k_0}^m \Ek$, and whose diagonal entries are the identity. Moreover,
\[
\innprodSS{u}{v} =\rho^{2m} \innprodTT{J u}{J v},
\qquad
u,v\in C^\infty(\Xone;\E).
\]
Finally, 
\[
\LsectionsS(\Xone;\E) = \rho^{\frac n2 - m + 1} J^{-1} \LsectionsT(\Xone;\E).
\]
\end{lemma}
\begin{proof}
As $\t=\s/\rho$, the differentials are related by
\[
\dss = \dtt + \drr
\]
hence by the binomial expansion
\[
a_k (\dss)^{m-k}\sp  \pi^* u^{(k)} = \sum_{j=0}^{m-k} a_{k+j} (\dtt)^{m-k-j} \sp \textstyle\binom{m-k}{j} \frac{a_k}{a_{k+j}} (\drr)^j  \sp \pi^* u^{(k)}.
\]
where $u^{(k)}\in C^\infty(\Xone;\Ek)$. This defines the endomorphism $J$ by declaring
\[
J u^{(k)} = \sum_{j=0}^{m-k} \textstyle\binom{m-k}{j} \frac{a_k}{a_{k+j}} (\drr)^j  \sp u^{(k)}.
\]
The second claim is direct from $\s^{-2}\eta=\rho^{-2}\t^{-2}\eta$, hence on $\F$, where the inner product requires $m$ applications of the inverse metric, $\innprodSETA{\cdot}{\cdot} = \rho^{2m}\innprodTETA{\cdot}{\cdot}$. The final claim follows from the second claim and the previously remarked correspondence, $\meas = \rho^{n+2} d\vol_g$.
\end{proof}

%%%%% %%%%% %%%%% %%%%% %%%%% %%%%% %%%%% %%%%% %%%%% %%%%% %%%%% %%%%%
%%%%% %%%%% %%%%% %%%%% %%%%% %%%%% %%%%% %%%%% %%%%% %%%%% %%%%% %%%%%
%%%%% %%%%% %%%%% %%%%% %%%%% %%%%% %%%%% %%%%% %%%%% %%%%% %%%%% %%%%%

\section{$\mathrm{b}$-Calculus and Microlocal Analysis}\label{sec:bcalc:microlocal}

This section introduces the necessary b-calculus formalism on symmetric cotensors. The standard reference is \cite{melrose:aps}, in particular we make much use of Chapters 2 and 5. We also recall some now standard ideas from microlocal analysis.

\subsection{b-calculus}\label{subsec:b-calculus}

For convenience we will only work on $\Mone=\RplusS\times \Xone$ rather than on both $\Mone$ and $\Mthree$.  We define $\overline\Mone$ to be the closure of $\Mone$ seen as a submanifold of $\R_\s\times \Xone$ with its usual topology. Then
\[
\overline\Mone = \Mone \sqcup \Xone
\]
where $\Xone$ is naturally identified with the boundary $\partial\overline\Mone=\{\s=0\}$.

%\begin{remark}
%Considering the model geometry, which motivates the viewpoint of hyperbolic space ``at infinity" inside the forward light cone of compactified Minkowski space, it would be somewhat more natural to introduce the coordinate $\s'=\s^{-1}$ on $\Mone$, then construct the closure of $\Mone$ as a submanifold of $\R_{\s'}\times\Xone$. The aesthetics of such a choice are outweighed by the superfluous introduction of two dual variables, one for each of $\s$ and $\t$.
%\end{remark}

We let $\{e_i\}_{i=0}^n$ denote a (local) holonomic frame for $\T\Xone$ and $\{e^i\}_{i=0}^n$ its dual frame for $\T^*\Xone$. The Lie algebra of b-vector fields consists of smooth vector fields on $\overline\Mone$ tangent to the boundary $\Xone$. It is thus generated by $\{\sds, e_i\}$. This provides the smooth vector bundle $\bT \overline\Mone$. The dual bundle, $\bT^*\overline\Mone$, has basis $\{ \dss, e^i \}$. This dual bundle is used to construct the b-symmetric bundle of $m$-cotensors, denoted $\bF$. On the interior of $\overline\Mone$, this bundle is canonically isomorphic to $\F$. 

An operator $\QVasyUp$ belongs to $\Diff^p_\bbb(\overline\Mone;\End\bF)$ if, relative to a frame generated by $\{\dss,e^i\}$ the operator $\QVasyUp$ may be written as a matrix
\[
\QVasyUp=[\QVasyUp_{i,j}]
\]
whose coefficients $\QVasyUp_{i,j}$ belong to $\Diff^p_\bbb(\overline\Mone)$. That is, each $\QVasyUp_{i,j}$ may be written
\[
\QVasyUp_{i,j} = \sum_{k,|\alpha|\le p} q_{i,j,k,\alpha} (\sds)^k  \partial_{x}^\alpha 
\]
for smooth functions $q_{i,j,k,\alpha}\in C^\infty(\overline\Mone)$. 

Operators in $\Diff^p_\bbb(\overline\Mone;\End\bF)$ provide indicial families of operators belonging to $\Diff^p(\Xone;\End\E)$. In order to define this mapping we recall the operator $\pi_{\s=\s_0}$ defined in the previous section for $\s_0\in\R^+$. This family of maps clearly has an extension to $\overline\Mone$ giving 
\[
\pi_{\s=\s_0} : 
		C^\infty(\overline\Mone; \bF) \to C^\infty (\Xone; \E)
\]
where $\s_0\in [0,\infty)$. The indicial family mapping (with respect to the Minkowski scale $\s$)
\[
\IF_\s : \Diff^p_\bbb(\overline\Mone;\End\bF) \to \mathcal{O}(\C; \Diff^p(\Xone;\End\E)).
\]
is defined by
\[
\IF_\s (\QVasyUp,\lambda) \left( u \right) 
:= 
\pi_{\s=0} \left(
				\s^{\lambda} \QVasyUp \s^{-\lambda} \left(
													\pistarS u
										\right)
		\right),
\qquad
u\in C^\infty(\Xone;\E).
\]
When the scale $\s$ is understood, we will use the convention of removing the bold font from such an operator and write
\[
\QVasyDown := \IF_\s(\QVasyUp,\cdot),
\qquad
\QVasyDOWNL := \IF_\s(\QVasyUp,\lambda).
\]

\begin{remark}
This definition effectively does three things. First, if $\QVasyUp$ is written as a matrix, relative to the decomposition established by the Minkowski scale \eqref{eqn:decomp:mink}, then $\QVasyDown$ will take the same form but without the appearances of $a_k(\dss)^{m-k}\sp$. Next, the functions $q_{i,j,k,\alpha}$ are frozen to their values at $\s=0$. (These two results are due to the appearance of $\pi_{\s=0}$.) Finally, due to the conjugation by $\s^{\lambda}$, all appearances of $\sds$ in $\QVasyUp$ are replaced by the complex parameter $-\lambda$.
\end{remark}

\begin{remark} The choice to conjugate by $\s^{-\lambda}$ is to ensure that the subsequent operators (in particular $\PVasyDown$) acting on $\Lsections$ sections, have physical domains corresponding to $\Re\lambda\gg1$. If one is convinced that the convention ought to be conjugation by $\s^\lambda$ rather than $\s^{-\lambda}$ one can kill two birds with one stone: Considering the model geometry, which motivates the viewpoint of hyperbolic space ``at infinity'' inside the forward light cone of compactified Minkowski space, it would be somewhat more natural to introduce the coordinate $\tilde\s=\s^{-1}$ on $\Mone$, then construct the closure of $\Mone$ as a submanifold of $\R_{\tilde\s}\times\Xone$. The indicial family would then by constructed via a conjugation of $\tilde\s^\lambda$ and appearances of $\tilde \s \partial_{\tilde \s}=-\sds$ would be replaced by $\lambda$. For this article, the aesthetics of such a choice are outweighed by the superfluous introduction of two dual variables, one for each of $\s$ and $\t$.\end{remark}

The b-operators we consider are somewhat simpler than the previous definition in that the coefficients $q_{i,j,k,\alpha}$ do not depend on $\s$ (in the correct basis).
\begin{definition}
A b-operator $\QVasyUp\in \Diff^p_\bbb(\overline\Mone; \bF)$ is b-trivial if, for all $\s_0\in\R^+$,
\[
\IF_\s (\QVasyUp,\lambda) \left( u \right) 
= 
\pi_{\s=\s_0} \left(
				\s^{\lambda} \QVasyUp \s^{-\lambda} \left(
													\pistarS u
										\right)
		\right),
\qquad
u\in C^\infty(\Xone;\E).
\]
\end{definition}
One advantage of this property is that self-adjointness of $\QVasyUp$  easily implies self-adjointness of $\QVasyDOWNL$ for $\lambda\in i\R$.

\begin{lemma}\label{lem:selfadjointpreservation}
Suppose $\QVasyUp$ is b-trivial and formally self-adjoint relative to the inner product
\[
\LinnprodSETA{u}{v} = \int_{\Mone} \innprodSETA{ u }{ v } \, \dss d\vol_g,
\qquad
u,v\in C^\infty_c(\Mone ; \F).
\]
Then, the indicial family $\QVasyDown$ is, upon restriction to $\lambda\in i\R$, formally self-adjoint relative to the inner product
\[
\LinnprodSS{u}{v} = \int_{\Xone} \innprodSS{ u }{ v } \, d\vol_g,
\qquad
u,v\in C^\infty_c(\Xone ; \E).
\]
Moreover, for all $\lambda$, $\QVasyDOWNL^* = \QVasyDown_{-\bar\lambda}$.
\end{lemma}
\begin{proof}
We prove only the first claim. That $\QVasyDOWNL^* = \QVasyDown_{-\bar\lambda}$ for all $\lambda$ follows by the same reasoning making the obvious changes in the second display provided below.
Let $\psi$ be a smooth function on $\RplusS$ with compact support (away from $\s=0$) and with unit mass $\int_{\R^+} \psi \, \dss =1$. Let $u,v\in C^\infty_c(\Xone ; \E)$. The b-triviality provides
\begin{align*}
\LinnprodSS{ \QVasyDOWNL u }{v }
&=
\int_{\R^+} \LinnprodSS{ \QVasyDOWNL u }{v } \,\psi\, \dss \\
&=
\LinnprodSETA{ \s^{\lambda} \QVasyUp \s^{-\lambda} \pistarS u }{ \psi \pistarS v}
\end{align*}
For $\lambda\in i\R$ this develops as
\begin{align*}
\LinnprodSS{ \QVasyDOWNL u }{v } &= \LinnprodSETA{ \pistarS u }{   \s^{\lambda} \QVasyUp \s^{-\lambda} \psi \pistarS v} \\
&=
\LinnprodSETA{ \pistarS u }{  \psi  \s^{\lambda} \QVasyUp \s^{-\lambda} \pistarS v} + 
\LinnprodSETA{ \pistarS u }{  [\s^{\lambda} \QVasyUp \s^{-\lambda} , \psi ] \pistarS v} \\
&=
\LinnprodSS{ u }{   \QVasyDOWNL v} + 
\LinnprodSETA{ \pistarS u }{  [\s^{\lambda} \QVasyUp \s^{-\lambda} , \psi ] \pistarS v} 
\end{align*}
where the last line has again used the b-triviality. Thus we require
\begin{align}\label{eqn:firstshittyequation}
\LinnprodSETA{ \pistarS u }{  [\s^{\lambda} \QVasyUp \s^{-\lambda} , \psi ] \pistarS v} =0
\end{align}
Consider $\QVasyUp$ as a matrix $\QVasyUp=[\QVasyUp_{i,j}]$ with respect to a basis in which
\[
\QVasyUp_{i,j} = \sum_{k,|\alpha|\le p} q_{i,j,k,\alpha} (\sds)^k  \partial_{x}^\alpha  
\]
for $q_{i,j,k,\alpha}\in C^\infty(\Xone)$. The key is to note that we may write
\begin{align}\label{eqn:secondshittyequation}
[ \s^{\lambda} \QVasyUp_{i,j} \s^{-\lambda} , \psi ] 
=
\sum_{k,|\alpha|\le p-1} \kappa_{i,j,k,\alpha} (\sds)^k \partial_{x}^\alpha
\end{align}
for smooth functions (which depend on $\lambda$) $\kappa_{i,j,k,\alpha}\in C^\infty(\Xone)$ such that every term in each $\kappa_{i,j,k,\alpha}$ is smoothly divisible by some non-zero integer $\sds$-derivative of $\psi$. Factoring out these appearances and integrating over $\R^+$ in \eqref{eqn:firstshittyequation} causes, by the fundamental theorem of calculus, the problematic term to vanish. The factorisation claim involving the functions $\kappa_{i,j,k,\alpha}$ follows directly from the following calculation. First
\begin{align*}
[ \s^{\lambda} \QVasyUp_{i,j} \s^{-\lambda} , \psi ] 
&=
 \sum_{k,|\alpha|\le p} q_{i,j,k,\alpha} [ (\sds-\lambda)^k  \partial_{x}^\alpha  , \psi ] \\
 &=
  \sum_{\substack{k,|\alpha|\le p \\ k\ge1}} q_{i,j,k,\alpha} [ (\sds-\lambda)^k,\psi] \partial_{x}^\alpha
\end{align*}
and for $k>1$,
\begin{align*}
[ (\sds-\lambda)^k,\psi] 
&= \sum_{\ell=1}^k \textstyle\binom{k}{\ell}(-\lambda)^{k-\ell} [ (\sds)^\ell , \psi ] \\
&= \sum_{\ell=1}^k \sum_{m=1}^\ell  \textstyle\binom{k}{\ell}(-\lambda)^{k-\ell} \textstyle\binom{\ell}{m} ((\sds)^m\psi) (\sds)^{\ell-m}
\end{align*}
which, due to the appearance of $(\sds)^m\psi$ gives \eqref{eqn:secondshittyequation} with the desired structure.
\end{proof}
\begin{remark}
The use of $d\vol_g$ is unimportant, the result holds for any measure on $\Xone$ given such a measure also appears as $d\vol_g$ does in the inner product on $\Mone$. 
\end{remark}

We finish this subsection by remarking the effect that the scale (Minkowski or Euclidean) has on the indicial family.

\begin{lemma}\label{lem:indicialfamily:changeofscales}
For $\QVasyUp\in\Diff^p_\bbb(\overline\Mone;\bF)$, the indicial families obtained using the scales $\s$ and $\t$ are related by
\[
\IF_\s (\QVasyUp,\lambda) = \rho^{\lambda} J^{-1} \IF_\t(\QVasyUp,\lambda) J \rho^{-\lambda}
\]
with $J$ presented in Lemma~\ref{lem:innprod:changeofscales}.
\end{lemma}

\begin{proof}
Lemma~\ref{lem:innprod:changeofscales} provides $\pistarS = \pistarT\circ J$. Dual to this equation, $\pi_{\s=0}=J^{-1}\circ \pi_{\t=0}$.
Combining these observations gives the result
\begin{align*}
\IF_\s (\QVasyUp,\lambda) \left( u \right) 
&= 
\pi_{\s=0} \left(
				\s^{\lambda} \QVasyUp \s^{-\lambda} \left(
													\pistarS u
										\right)
		\right)\\
&=
J^{-1} \pi_{\t=0}\left(
				\rho^{\lambda}\t^{\lambda} \QVasyUp \t^{-\lambda}\rho^{-\lambda} \left(
													\pistarT J u
										\right)
		\right)\\
&=
\rho^{\lambda} J^{-1} \IF_\t(\QVasyUp,\lambda) (J \rho^{-\lambda} u ). \qedhere
\end{align*}
\end{proof}

\subsection{Microlocal analysis}

We recall standard objects in microlocal analysis (the necessary information is given in \cite{zworski:vm} for pseudodifferential operators acting on the trivial bundle, here we merely indicate the small changes that occur when acting on a vector bundle). 
Recall the open submanifold $\Xtwo=\{ \mu > -\tfrac12\}\subset\Xthree$ from Defintion~\ref{def:threebasemanifolds}.
We will assume that $\LsectionsT(\Xthree;\E)$ provides a notion of sections above $\Xtwo$ with Sobolev regularity $s$, denoted $H^s(\Xtwo;\E)$, with norm $\norm{\cdot}{H^s}$ (see Subsection~\ref{subsec:functionspaces} for subtleties arising due to the boundary $\CS$) . Let $\zeta$ denote the coefficients of a covector relative to some local base for $\T^*\Xtwo$ such that we may define the Japanese bracket $\japbrak{\zeta}$. We denote by
\[
\Psi^p_\scal (\Xtwo; \End\E) \subset \Psi^p(\Xtwo;\End\E)
\]
the space of properly supported pseudo-differential operators of order $p$ acting on $\E$ and which have scalar principal symbol. For $A\in \Psi^a_\scal (\Xtwo; \End\E)$ such a symbol is written
\[
\principalsymbol{A}\in S^a(\T^*\Xtwo \backslash 0 ; \End \E ) / S^{a-1}(\T^*\Xtwo\backslash 0 ; \End \E)
\]
and is scalar.
For such operators, it continues to hold that, for $B\in \Psi^b_\scal (\Xtwo; \End\E)$, the principal symbol of the composition
\[
\principalsymbol{AB} = \principalsymbol{A}\principalsymbol{B} \in S^{a+b}(\T^*\Xtwo \backslash 0 ; \End \E ) / S^{a+b-1}(\T^*\Xtwo\backslash 0 ; \End \E)
\]
remains scalar. However now, as lower order terms are not required to be diagonal, the commutator has principal symbol
\[
\principalsymbol{[A,B]} \in S^{a+b-1}(\T^*\Xtwo \backslash 0 ; \End \E ) / S^{a+b-2}(\T^*\Xtwo\backslash 0 ; \End \E)
\]
which, in general, is not scalar.
In the case that $A\in \Psi^a(\Xtwo)\subset\Psi^a_\scal(\Xtwo;\End\E)$ we get $\principalsymbol{\frac{1}{2i}[A,B]} = \frac12 \Hamiltonian{\principalsymbol{B}}(\principalsymbol{A})$ where $\Hamiltonian{\principalsymbol{B}}$ is the Hamiltonian vector field associated with $\principalsymbol{B}$.
Exactly as in the case that $\E$ is the trivial bundle, associated with the operator $A$ are the notions of the wave front set $\WF(A)$ and the characteristic variety $\Char(A)$.

There are two radial estimates used in the analysis of $\PVasyDown$ (the family of operators introduced in Section~\ref{sec:P}) in order to prove Proposition~\ref{prop:fredholm}. The analysis is performed in \cite[Section 2.4]{v:ml:inventiones} for functions with an alternative description given in \cite[E.5.2]{dyatlov-zworski:book}. We will follow the second approach and translate the results into a (non semiclassical) setting adapted to vector bundles. For this, and to follow closely the referenced works, we introduce \cite[Subsection E.1.2]{dyatlov-zworski:book} the radially compactified cotangent bundle $\rct\Xtwo$ and projection map $\kappa: \T^*\Xtwo\backslash0 \to \partial \rct\Xtwo$. Consider $P\in \Psi^p_\scal(\Xtwo;\End\E)$ with real principal symbol $\principalsymbol{P}$ and Hamiltonian vector field $\Hamiltonian{\principalsymbol{P}}$. Write $P$ in the following way
\[
P = \Re P + i \Im P
\]
for
\[
\Re P = \frac{P+P^*}{2} \in \Psi^p_\scal(\Xtwo;\End\E),
\qquad
\Im P = \frac{P-P^*}{2i} \in \Psi^{p-1}(\Xtwo;\End\E).
\]
In the sense of \cite[Definition E.52]{dyatlov-zworski:book}, let $\cSource$ and $\cSink$ be a source and a sink of $\principalsymbol{P}$ respectively. Suppose that $\japbrak{\zeta}^{1-p}\Hamiltonian{\principalsymbol{P}}$ vanishes on $\cSourceSink$.  Then
\begin{lemma}\label{lem:highregularity} 
Let $s$ satisfy the following threshold condition on $\cSource$ that
\[
\japbrak{\zeta}^{1-p}( \principalsymbol{\Im P} + (s + \tfrac{1-p}{2}) \Hamiltonian{\principalsymbol{P}} \log\japbrak{\zeta} ) \quad\textrm{is negative definite.}
\]
Then for all $B_1\in \Psi^0(\Xtwo)$ with $\WF(I-B_1)\cap \cSource=\varnothing$, there exists $A\in \Psi^0(\Xtwo)$ with $\Char(A)\cap \cSource=\varnothing$ such that for any $u\in C_c^\infty(\Xtwo;\E)$ (and any $N$ large enough)
\[
\norm{Au}{H^s} \le C( \norm{B_1Pu}{H^{s-p+1}} + \norm{u}{H^{-N}}).
\]
\end{lemma}
\begin{lemma}\label{lem:lowregularity}
Let $s$ satisfy the following threshold condition on $\cSink$
\[
\japbrak{\zeta}^{1-p} ( \principalsymbol{\Im P} + (s + \tfrac{1-p}{2}) \Hamiltonian{\principalsymbol{P}} \log\japbrak{\zeta} ) \quad\textrm{is negative definite.}
\]
Then for all $B_1\in \Psi^0(\Xtwo)$ with $\WF(I-B_1)\cap \cSink=\varnothing$, there exists $A,B\in \Psi^0(\Xtwo)$ with $\Char(A)\cap \cSink=\varnothing$ and $\WF(B)\cap\cSink=\varnothing$ such that for any $u\in C_c^\infty(\Xtwo;\E)$ (and any $N$ large enough)
\[
\norm{Au}{H^{s}} \le C( \norm{Bu}{H^s} + \norm{B_1Pu}{H^{s-p+1}} + \norm{u}{H^{-N}}).
\]
\end{lemma}
\begin{remark}\label{rem:highlowregularity}
There are two trivial but important points to make. First, a source for $P$ is a sink for $-P$ (and similarly a sink for $P$ is a source for $-P$). Second, we have assumed $P$ has real principal symbol therefore, when considering its adjoint $P^*$, we have $\Hamiltonian{\principalsymbol{P^*}}= \Hamiltonian{\principalsymbol{P}}$.
Less trivially, by approximation \cite[Lemma E.47]{dyatlov-zworski:book}, these results do not need to assume $u\in C_c^\infty(\Xtwo;\E)$. In Lemma~\ref{lem:highregularity}, if $s>\tilde s$ with $\tilde s$ satisfying the threshold condition and $u \in H^{\tilde s}(\Xtwo;\E)$ then the inequality holds (on the condition that the right hand side is finite). Similarly in Lemma~\ref{lem:lowregularity}, if $u$ is a distribution such that the right hand side of the inequality is well defined, then so too is the left hand side, and the inequality holds.
\end{remark}

%%%%% %%%%% %%%%% %%%%% %%%%% %%%%% %%%%% %%%%% %%%%% %%%%% %%%%% %%%%%
%%%%% %%%%% %%%%% %%%%% %%%%% %%%%% %%%%% %%%%% %%%%% %%%%% %%%%% %%%%%
%%%%% %%%%% %%%%% %%%%% %%%%% %%%%% %%%%% %%%%% %%%%% %%%%% %%%%% %%%%%

\section{The Laplacian, the d'Alembertian and the Operator $\QVasyUp$}\label{sec:Lap.DAlem.Q}

This section shows the relationship between the Laplacian on $(\Xone,g)$ and the d'Alembertian on $(M,\eta)$.  We first introduce several differential operators on $\Xone$ using the Levi-Civita connection $\n$ of $g$ extended to all associated vector bundles associated with the principal orthonormal frame bundle. Let $\{e_i\}_{i=0}^n$ be a local orthonormal frame for $\T\Xone$ and $\{e^i\}_{i=0}^n$ the corresponding dual frame for $\T^*\Xone$. We define two first-order differential operators. Let the symmetrisation of the covariant derivative, called the symmetric differential, be denoted $\d$:
\[
\d : \left\{	 \begin{array}{rcl}
			C^\infty(\Xone ; \Ek) & \to & C^\infty(\Xone ; \E^{(k+1)}) \\
			u & \mapsto & \sum_{i=0}^n e^i\sp \n_{e_i} u
		\end{array}
		\right.
\]
and, by $\div$, its formal adjoint, called the divergence:
\[
\div : \left\{	 \begin{array}{rcl}
			C^\infty(\Xone ; \Ek) & \to & C^\infty(\Xone ; \E^{(k-1)}) \\
			u & \mapsto & -\sum_{i=0}^n \hookprod{e^i}{\n_{e_i} u}
		\end{array}
		\right.
\]
The two first order-operators behave nicely with $\adjtrace$ and $\trace$ giving the following commutation relations \cite[Equation 8]{hms}:
\begin{align}\label{eqn:HMScommutations}
[\trace,\div]=0=[\adjtrace,\d],
\quad
[\trace,\d]=-2\div,
\quad
[\adjtrace,\div]=2\d.
\end{align}
The rough Laplacian on this space will be denoted by $\n^*\n$:
\[
\n^*\n : \left\{	 \begin{array}{rcl}
			C^\infty(\Xone ; \Ek) & \to & C^\infty(\Xone ; \Ek) \\
			u & \mapsto &  \n^*\n u
		\end{array}
		\right.
\]
where $\n^*$ is the formal adjoint of $\n : C^\infty(\Xone ; \Ek) \to C^\infty(\Xone ; \T^*\Xone\otimes \Ek)$. Equivalently
\[
\n^*\n \,u = (-\tr \circ \n \circ \n) (u),
\qquad
u\in C^\infty(\Xone;\Ek)
\]
where $\tr : \T^*\Xone\otimes \T^*\Xone \to\R$ is the trace operator obtained from $g$ and is extended to $\tr : \T^*\Xone\otimes \T^*\Xone\otimes \Ek \to \Ek$. For the Lichnerowicz Laplacian, we introduce the Riemann curvature tensor which will be denoted by $\Riemann$:
\[
\Riemann_{{u},{v}}{w} = [\n_{u},\n_{v}] w - \n_{[u,v]}w,
\qquad
u,v,w\in C^\infty(\Xone; TX)
\]
and is extended to all tensor bundles as a derivation. On symmetric $k$-cotensors we introduce the curvature endomorphism which will be denoted by $q(\Riemann)$:
\[
q(\Riemann) \, u = \sum_{i,j=0}^{n} e^j\sp \hookprod{e^i}{\Riemann_{e_i,e_j} u},
\qquad
u\in \Ek.
\]
The Lichnerowicz Laplacian, hereafter simply referred to as the Laplacian, will be denoted by $\Delta$:
\[
\Delta : \left\{	 \begin{array}{rcl}
			C^\infty(\Xone ; \Ek) & \to & C^\infty(\Xone ; \Ek) \\
			u & \mapsto &  (\n^*\n + q(\Riemann) ) u
		\end{array}
		\right.
\]
We decompose symmetric $k$-cotensors using the symmetrised basis elements:
\[
u = \sum_{K\in\sA^k} u_K e^K,
\qquad
u\in C^\infty(\Xone; \Ek),
u_K \in C^\infty(\Xone).
\]
Useful formulae for the preceding operators thus far introduced are given in the following lemma. Recall the notation for finite sequences $\sA^k$ introduced in the final paragraph of Subsection~\ref{subsec:singlefibre}.

\begin{lemma}\label{lem:simpleformulae}
Let $u\in C^\infty(\Xone; \Ek)$. At a point in $\Xone$ about which $\{e_i\}$ are normal coordinates, the trace is
\[
\trace u
=
\sum_{K\in\sA^k}
\sum_{k_r\in K} \sum_{k_{p} \in \co{k_r}{}{K}} 
g^{k_r k_p}u_K e^\ct{k_p}{}{k_r}{}{K},
\]
the symmetric differential is
\[
\d u 
=
 \sum_{K\in\sA^k}	 \sum_{i=0}^n
  (e_i u_K) e^\co{}{i}{K},
\]
the divergence is
\[
 \div  u 
=
- \sum_{K\in\sA^k} \sum_{k_r\in K} \sum_{i=0}^n g^{ik_r} (e_i u_K) e^\co{k_r}{}{K},
\]
the rough Laplacian is
\[
\n^*\n \, u
=
\sum_{K\in \sA^k}	
		\sum_{i,j=0}^n
 	\left(
		-g^{ij} (e_i e_j u_K) e^K
 		+
		\sum_{k_r\in K}
		\sum_{\ell=0}^n
		g^{i\ell} u_K ( {e_\ell} \chris{ij}{k_r}) e^\co{k_r}{j}{K}
	\right)
\]
where the connection coefficients are given locally by $\n_{e_i} e^k = - \sum_{j=0}^n \chris{ij}{k}e^j$. Finally, (at a point using normal coordinates), the Riemann curvature takes the form
\[
\Riemann_{e_i,e_j} e^\ell = -\sum_{k=0}^n \tensor{{\Riemann}}{_{ij}^{\ell}_{k}} e^k,
\quad
\tensor{{\Riemann}}{_{ij}^{\ell}_{k}} = {e_i} \chris{jk}{\ell} - {e_j} \chris{ik}{\ell}.
\]
\end{lemma}

Similar in vein to \eqref{eqn:HMScommutations} we have the following two useful results, the second of which originates from \cite[Section 10]{lichnerowicz}.
\begin{lemma}\label{lem:laplaciancommutestracediv}
Let $u\in C^\infty (\Xone,\Ek)$ The Laplacian commutes with the Lefschetz-type trace operator
\[
[ \trace, \Delta ] u =0
\]
and commutes with the divergence under the following conditions
\[
[  \div , \Delta ] u =0 \textrm{ if } \left\{
								\begin{array}{l}
												k=0,1, \textrm{ } \\
												k=2 \textrm{ and $\Xone$ is Ricci parallel,} \\
												k\ge 3 \textrm{ and $\Xone$ is locally isomorphic to $\Hn$.}
								\end{array}
						\right.
\]
 \end{lemma}

\begin{proof}
The first result is very standard. As the metric is parallel, the Riemann curvature tensor (acting as a derivation on $\Ek$) commutes with $\adjtrace$ hence 
\[
[\adjtrace,q(\Riemann)]u = \sum_{i,j=0}^n ( \adjtrace \hookprod{e^j}{\hookprod{e^i}{}}- \hookprod{e^j}{\hookprod{e^i}{}}\adjtrace) \, \Riemann_{e_i, e_j} u
\]
and developing the second term with the aid of the commutation formula $[\hookprod{e^i}{},\adjtrace]= 2 e^i \sp $ provides
\begin{align*}
[\adjtrace,q(\Riemann)]u 
	&= \sum_{i,j=0}^n -2(e^j \sp \hookprod{e^i}{} + \hookprod{e^j}{e^i\sp}\,) \, \Riemann_{e_i,e_j} u \\
	&= \sum_{i,j=0}^n -2(e^j \sp \hookprod{e^i}{} + \delta^{ij} + e^i \sp \hookprod{e^j}{}) \, \Riemann_{e_i,e_j} u
\end{align*}
which vanishes due to the skew-symmetry of the Riemann curvature tensor. By duality, $[\trace,q(\Riemann)]=0$. Now using the commutation relations \eqref{eqn:HMScommutations} and the following characterisation of the Laplacian \cite[Proposition 6.2]{hms}
\[
\Delta = \div \d - \d\div + 2q(\Riemann)
\]
provides the commutation of $\trace$ with $\Delta$.

The second result is more involved as a demonstration via a direct calculation (however as these statements are well known, we only sketch said calculations). For $k=0,1$ the Laplacian and divergence agree with Hodge Laplacian and the adjoint of the exterior derivative. We will thus assume $\Xone$ is Ricci parallel (and $k\ge 2$). We break the calculation into two parts studying $[ \div ,\n^*\n]$ and $[ \div ,q(\Riemann)]$. As usual, we use a frame $\{e_i\}_{i=0}^n$ for $\T\Xone$ with dual frame $\{e^i\}_{i=0}^n$ and calculate at a point about which the connection coefficients vanish. We act on $u=u_Ke^K\in C^\infty(\Xone;\Ek)$. That the Ricci tensor is parallel implies, by the (second) Bianchi identity, $\sum_\ell \n_{e_\ell} \tensor{{\Riemann}}{_{ij}^{\ell}_{k}}=0$. This observation is repeatedly used. Also, the Ricci endomorphism may be written
$\sum_{i,j} \Ric_i^j e^i\otimes e_j$
 with
$\Ric_i^j = \sum_{k,\ell}g^{k\ell}( \n_{e_i} \chris{k\ell}{j} - \n_{e_k}\chris{\ell i}{j})$.

Consider $[ \div ,\n^*\n]$. Calculating simply $ \div  \n^*\n$ gives
\begin{align*}
\div \n^*\n  &= -\tsum_{k} \hookprod{e^k}{\n_{e_k}} ( -\tr \tsum_{i,j}  e^i \otimes \n_{e_i} ( e^j \otimes \n_{e_j})) \\
&= \tsum_{i,j,k} g^{ij} \hookprod{e^k}{\n_{e_k}\n_{e_i}\n_{e_j}} - \tsum_{i,j,k,\ell} g^{i\ell} (\n_{e_k} \chris{i\ell}{j} )\,\hookprod{e^k}{\n_{e_j}}
\end{align*}
with a similar calculation for $\n^*\n \div $. Combining these results and commuting $\n_{e_k}$ with $\n_{e_i}\n_{e_j}$ gives
\begin{align*}
[ \div , \n^*\n ] & = \tsum_{i,j,k} g^{ij} \hookprod{e^k}{[ \n_{e_k} , \n_{e_i}\n_{e_j}]}- \tsum_i \hookprod{(\Ric  e^i)}{\n_{e_i}} \\
&= -\tsum_{i,j,k} g^{ij} \hookprod{e^k}{ \{ \n_{e_i} , \Riemann_{e_j,e_k} \} } - \tsum_i\hookprod{(\Ric  e^i)}{\n_{e_i}}
\end{align*}
where $\{\cdot,\cdot\}$ is the anticommutator. After a tedious calculation, we obtain
\begin{align}\label{eqn:deltalaprough}
[ \div , \n^*\n ] u = \sum_i\hookprod{(\Ric  e^i)}{\n_{e_i}u} + 2(\Riemann,\n,u)
\end{align}
where $(\Riemann,\n,u)$ is shorthand for the unwieldy term
\[
(\Riemann,\n,u)
=
\sum_{i,j} \sum_{k_r\in K} \sum_{k_p\in\co{k_r}{}{K}} \tensor{\Riemann}{^{i k_r k_p}_{j}} (\n_{e_i} u_K) e^\ct{k_p}{j}{k_r}{}{K}.
\]
For completeness we outline this calculation
\begin{align*}
-\tsum_{i,j,\ell} g^{ij} \hookprod{e^\ell}{ \{ \n_{e_i} , \Riemann_{e_j,e_\ell} \} } u
&= -\tsum_{i,j,\ell} \tsum_{k_r\in K} ( \{ \n_{e_i} , \tensor{\Riemann}{_\ell ^{ik_r} _j } \} u_K) \hookprod{e^\ell}{e^\co{k_r}{j}{K}} \\
&= -2\tsum_{i,j,\ell} \tsum_{k_r\in K} \tensor{\Riemann}{_\ell ^{ik_r} _j } (\n_{e_i} u_K) \hookprod{e^\ell}{e^\co{k_r}{j}{K}}
\end{align*}
where the anticommutator has been removed using $\sum_\ell \n_{e_\ell} \tensor{{\Riemann}}{_{ij}^{\ell}_{k}}=0$. Developing the final term in the preceding display gives
\begin{align*}
\hookprod{e^\ell}{e^\co{k_r}{j}{K}} = g^{j\ell} e^\co{k_r}{}{K} + \tsum_{k_p\in\co{k_r}{}{K}} g^{k_p\ell} e^\ct{k_p}{j}{k_r}{}{K}
\end{align*}
which after a little rearrangement of dummy indices and using the algebraic symmetries of the Riemann curvature tensor gives
\[
-\tsum_{i,j,\ell} g^{ij} \hookprod{e^\ell}{ \{ \n_{e_i} , \Riemann_{e_j,e_\ell} \} } u = 2\sum_i\hookprod{(\Ric  e^i)}{\n_{e_i}u} + 2(\Riemann,\n,u)
\]
upon subtraction of $\tsum_i\hookprod{(\Ric  e^i)}{\n_{e_i}u}$, this provides \eqref{eqn:deltalaprough}.

Consider $[ \div ,q(\Riemann)]$. Similar to the previous calculations we obtain
\begin{align*}
[ \div , q(\Riemann) ] 
&= \tsum_{i,j,k}-\hookprod{e^k}{e^j\sp \hookprod{e^i}{\n_{e_k}\Riemann_{e_i,e_j}}} + e^j\sp \hookprod{e^i}{\Riemann_{e_i,e_j} (\hookprod{e^k}{\n_{e_k}}}) \\
&= \tsum_{i,j,k} e^j\sp \hookprod{e^i}{\hookprod{e^k}{[\Riemann_{e_i,e_j}, \n_{e_k}]}} - g^{jk} \hookprod{e^i}{\n_{e_k} \Riemann_{e_i,e_j}} + e^j\sp \hookprod{e^i}{\hookprod{(\Riemann_{e_i,e_j} e^k) } \n_{e_k}}
\end{align*}
After an even more tedious calculation treating each of the three terms in the previous display, we obtain
\begin{align}\label{eqn:deltacurv}
[ \div , q(\Riemann) ] u = - [ \div ,\n^*\n]u - (\n,\Riemann,u)
\end{align}
where $(\n,\Riemann,u)$ represents the even more unwieldy term
\[
(\n,\Riemann,u) = \sum_{i,j,\ell} \sum_{ \substack{ k_r\in K \\ k_p \in \co{k_r}{}{K} \\ k_s \in \ct{k_p}{}{k_r}{}{K} }} g^{\ell k_s} (\n_{e_\ell} \tensor{\Riemann}{_i ^{k_rk_p} _j}) u_K e^\cthree{k_s}{i}{k_p}{j}{k_r}{}{K}.
\]
Again, we sketch the calculation. One of the three terms is easy to calculate directly giving
\[
\tsum_{i,j,k} e^j\sp \hookprod{e^i}{\hookprod{(\Riemann_{e_i,e_j} e^k) } \n_{e_k}} u = -(\Riemann,\n,u).
\]
Another term is also relatively easy, again using the trick that $\sum_\ell \n_{e_\ell} \tensor{{\Riemann}}{_{ij}^{\ell}_{k}}=0$,
\[
- \tsum_{i,j,k}g^{jk} \hookprod{e^i}{\n_{e_k} \Riemann_{e_i,e_j}} u=-\tsum_i\hookprod{(\Ric  e^i)}{\n_{e_i}u} -(\Riemann,\n,u)
\]
The involved step is treating 
$\tsum_{i,j,k} e^j\sp \hookprod{e^i}{\hookprod{e^k}{[\Riemann_{e_i,e_j}, \n_{e_k}]}}$.
We first obtain
\[
\tsum_{i,j,\ell} e^j\sp \hookprod{e^i}{\hookprod{e^\ell}{[\Riemann_{e_i,e_j}, \n_{e_\ell}]}} u 
=
\tsum_{i,k,\ell,m}\sum_{k_r\in K}([\tensor{\Riemann}{_{ji}^{k_r}_m},\n_{e_\ell}]u_K) \, e^j\sp \hookprod{e^i}{\hookprod{e^\ell}{e^\co{k_r}{m}{K}}}
\]
and it is important to realise that whenever the index $\ell$ contracts with $m$ (or $i$ or $j$), the resulting sum vanishes (as $\sum_\ell \n_{e_\ell} \tensor{{\Riemann}}{_{ij}^{\ell}_{k}}=0$). Similarly, if $i$ and $m$ are contracted then, as Ricci is parallel, the resulting sum vanishes. Expanding the final part of the previous display (and letting $\terms(g^{\ell m}, g^{im})$ denote any terms involving $g^{\ell m}$ or $g^{im}$) gives
\begin{align*}
e^j\sp \hookprod{e^i}{\hookprod{e^\ell}{e^\co{k_r}{m}{K}}}
&= \sum_{k_p\in\co{k_r}{}{K}} g^{\ell k_p} e^j \sp \hookprod{e^i}{ e^\ct{k_p}{m}{k_r}{}{K} } + \terms(g^{\ell m}) \\
&= \sum_{ \substack{ k_p\in\co{k_r}{}{K} \\ k_s\in \ct{k_p}{}{k_r}{}{K}} }	g^{\ell k_p} g^{i k_s} e^\cthree{k_s}{j}{k_p}{m}{k_r}{}{K}						+ \terms({g^{\ell m}, g^{i m}})
\end{align*}
and after a little rearrangement of dummy indices, this gives
\[
\tsum_{i,j,\ell} e^j\sp \hookprod{e^i}{\hookprod{e^\ell}{[\Riemann_{e_i,e_j}, \n_{e_\ell}]}} u = -(\n,\Riemann,u)
\]
whence \eqref{eqn:deltacurv} is obtained.

Combining \eqref{eqn:deltalaprough} with \eqref{eqn:deltacurv} gives $[ \div ,\Delta]u=-(\n,\Riemann,u)$. For symmetric tensors of rank two, such a summation (over $k_r,k_p,k_s$) does not arrive so such a term instantly vanishes and the result follows. For tensors of higher rank, one needs the Riemann curvature to be parallel. This is assured in the constant curvature setting of $\Hn$.
\end{proof}

The objects thus far introduced in this section all have natural analogues in the Lorentzian setting on $(\Mone,\eta)$. We denote by $\nM$ the Levi-Civita connection of $\eta$ extended to all associated vector bundles and $\MRiemann$ the Riemann curvature tensor of $\eta$. We let $\dETA$ and $\divETA$ denote the symmetric differential and the divergence with respect to $\eta$. Finally we let $\nM^*\nM$ denote the rough d'Alembertian and $\DeltaAmb$ the (Lichnerowicz) d'Alembertian both constructed with respect to the metric $\eta$.

\subsection{Minkowski scale and the operator $\QVasyUp$}

We define the first of our two main operators.
\begin{definition}
The second-order differential operator $\QVasyUp\in\Diff^2(\Mone; \End\F)$ is the following conjugation of the d'Alembertian:
\[
\QVasyUp : \left\{	 \begin{array}{rcl}
			C^\infty(\Mone ; \F) & \to & C^\infty(\Mone ; \F) \\
			u & \mapsto &  \s^{\frac n2-m +2} \DeltaAmb \s^{-\frac n2+m} \, u
		\end{array}
		\right.
\]
\end{definition}
\begin{lemma}\label{lem:Qselfadjoint}
The differential operator $\QVasyUp$ is formally self-adjoint with respect to the inner product
\[
\LinnprodSETA{u}{v} = \int_{\Mone} \innprodSETA{ u }{ v } \, \dss d\vol_g,
\qquad
u,v\in C^\infty_c (\Mone; \F).
\]
\end{lemma}
\begin{proof}
The d'Alembertian is self-adjoint with respect to the following inner product
\[
\LinnprodETA{ u}{ v } = \int_{\Mone} \innprodETA{ u }{ v } \, d\vol_\eta,
\qquad
u,v\in C^\infty_c (\Mone; \F).
\]

The two inner products on $\F$ are related via \eqref{eqn:twoinnprods:eta}. Tracking the effects of the conjugations by powers of $\s$ on $\DeltaAmb$ as well as the multiplication by $\s^2$ in order to obtain $\QVasyUp$ implies self-adjointness when using the inner product $\innprodSETA{\cdot}{\cdot}$ with the measure $\s^{-(n+2)}d\vol_\eta$ which gives the result as $d\vol_\eta = \s^{n+2}\dss  d\vol_g$.
\end{proof}

\begin{lemma}\label{lem:Qcommutestrace}
The operator $\QVasyUp$ commutes with the Lefschetz-type trace operator $\s^{-2}\traceSETA$.
\[
[ \s^{-2}\traceSETA, \QVasyUp ] u = 0 ,
\qquad
u\in C^\infty(\Mone;\F).
\]
\end{lemma}
\begin{proof}
The Lorentzian analogue of Lemma~\ref{lem:laplaciancommutestracediv} is that the d'Alembertian commutes with $\traceETA$
\[
[\traceETA,\DeltaAmb]=0.
\]
This operator is related to our standard Lefschetz-type operator $\traceSETA$ via \eqref{eqn:twoLambdas:eta}. The result is now a direct calculation. For clarity we denote differential operators with a superscript $(m)$ to indicate that they act on symmetric cotensors of rank $m$. In particular, on $C^\infty(\Mone;\F)$ we have
\begin{align*}
\s^{-2}\traceSETA \QVasyUp^{(m)}
&= \s^2  \traceETA \s^{\frac n2-m+2} \DeltaAmb^{(m)} \s^{-\frac n2+m} \\
&= \s^2\s^{\frac n2-m+2} \DeltaAmb^{(m-2)} \s^{-\frac n2+m}  \traceETA \\
&= \s^{\frac n2-(m-2)+ 2} \DeltaAmb^{(m-2)} \s^{-\frac n2 + (m-2)}\s^2 \traceETA \\
&= \QVasyUp^{(m-2)} \s^{-2} \traceSETA.\qedhere
\end{align*}
\end{proof}

The rest of this subsection is dedicated to proving
\begin{proposition}\label{prop:decompQ:k:mink}
For $u\in C^\infty(\Mone; \F)$ decomposed relative to the Minkowski scale \eqref{eqn:decomp:mink}, the conjugated d'Alembertian $\QVasyUp$ is given by
\begin{align*}
a_{k+2}\,(\dss)^{m-k-2}&\sp\big(
					-b_kb_{k+1} \adjtrace
		\big)u^{(k)} +\\
a_{k+1}\,(\dss)^{m-k-1}&\sp\big(
					2b_k\d
		\big)u^{(k)} +\\
 \QVasyUp   a_k\, (\dss)^{m-k} \sp  u^{(k)} \,=\,
a_{k\phantom{+0}}\,(\dss)^{m-k\phantom{+0}}&\sp\big(
					\Delta + (\sds)^2 -  c_k -  \adjtrace\trace
		\big)u^{(k)} +\\
a_{k-1}\,(\dss)^{m-k+1}&\sp\big(
					-2b_{k-1}  \div 
		\big)u^{(k)} +\\
a_{k-2}\,(\dss)^{m-k+2}&\sp\big(
					-b_{k-2}b_{k-1} \trace
		\big)u^{(k)}
\end{align*}
with constants
\begin{align*}
a_k &= ((m-k)!)^{-1/2}, \\
b_k & = \sqrt{m-k}, \\
c_k &= \tfrac{n^2}{4} + m(n+2k+1) - k(2n+3k-1).
\end{align*}
Consequently, relative to this scale, there exist $\DVasyUp\in \Diff^1(\Mone;\End\F)$ and $\ZeroOrderVasyUp\in C^\infty(\Mone;\End\F)$ independent of $\s$ such that
\[
\QVasyUp = \n^*\n + (\sds)^2 + \DVasyUp + \ZeroOrderVasyUp
\]
\end{proposition}
\begin{proof}
The result will follow from Lemmas~\ref{lem:decompk:rough:mink} and~\ref{lem:decompk:curvature}. The conjugation by $\s^{-\frac n2+m}$ is chosen so that the term
$(\sds+\frac n2 - m)^2$ in Lemma~\ref{lem:decompk:rough:mink} becomes simply $(\sds)^2$.
\end{proof}

Proposition~\ref{prop:decompQ:k:mink} is a direct calculation which we present in the rest of this subsection. To begin we announce the following lemma whose proof need not be detailed.
\begin{lemma}\label{lem:coneconnection}
In the Minkowski scale, with $\{e_i\}_{i=0}^n$ a local holonomic frame on $(\Xone,g)$ with dual frame $\{e^i\}_{i=0}^n$ such that $g=\sum_{i,j} g_{ij} e^i\otimes e^j$, the connection $\nM$ acts in the following manner:
\begin{align*}
\nM_{\sds} \dss &= - \dss,		
								&		\nM_{e_i} \dss &= -\textstyle\sum_{j=0}^n g_{ij} e^j, \\
\nM_{\sds} e^i &= -e^i	,	
								&		\nM_{e_i}e^j &= \delta_i^j \dss + \n_{e_i}e^j.
\end{align*}
\end{lemma}
This lemma provides the following two important formulae for the symmetrised basis
\begin{align}\label{eqn:derivationbasis:tau}
\nM_{\sds} (\dss)^{m-k}\sp e^K 
=
-m (\dss)^{m-k}\sp e^K
\end{align}
and
\begin{align}\label{eqn:derivationbasis:X}
(\dss)^{m-k-1}&\sp\left(
					-(m-k) g_{ij} e^\co{}{j}{K}
		\right) +\nonumber\\
\nM_{e_i} (\dss)^{m-k}\sp e^K  \,=\,
(\dss)^{m-k\phantom{+0}}&\sp\left(
					-\textstyle\sum_{k_r\in K} \chris{ij}{k_r} e^\co{k_r}{j}{K}
		\right)+\\
(\dss)^{m-k+1}&\sp\left(
					\textstyle\sum_{k_r\in K} \delta_i^{k_r} e^\co{k_r}{}{K}
		\right)\nonumber
\end{align}
where the second result is a consequence of
\[
\nM_{e_i} e^K =\sum_{k_r\in K} \delta_i^{k_r} \dss\sp e^\co{k_r}{}{K} + \n_{e_j} e^K
\]
and we recall that the connection coefficients were introduced in Lemma~\ref{lem:simpleformulae}. We split the calculation of the d'Alembertian into two calculations, treating the rough d'Alembertian separately from the curvature endomorphism.

\begin{lemma}\label{lem:decompk:rough:mink}
For $u\in C^\infty(\Mone;\F)$ decomposed relative to the Minkowski scale \eqref{eqn:decomp:mink}, the rough d'Alembertian is given by
\begin{align*}
a_{k+2}\,(\dss)^{m-k-2}&\sp\big(
					-b_{k}b_{k+1} \adjtrace
		\big)u^{(k)} +\\
a_{k+1}\,(\dss)^{m-k-1}&\sp\big(
					2b_k\d
		\big)u^{(k)} +\\
\s^2 \,\nM^*\nM \,  a_k\, (\dss)^{m-k} \sp  u^{(k)} \,=\,
a_{k\phantom{+0}}\,(\dss)^{m-k\phantom{+0}}&\sp\big(
					\n^*\n + (\sds + \tfrac n2 - m)^2 - \tilde c_k
		\big)u^{(k)} +\\
					a_{k-1}\,(\dss)^{m-k+1}&\sp\big(
					-2b_{k-1}  \div 
		\big)u^{(k)} +\\
a_{k-2}\,(\dss)^{m-k+2}&\sp\big(
					-b_{k-2}b_{k-1}\trace
		\big)u^{(k)}
\end{align*}
with modified constants
\[
\tilde c_k = \tfrac{n^2}{4} + m(n+2k+1) - k(n+2k).
\]
\end{lemma}
\begin{proof}
It suffices to consider a single term $u_K (\dss)^{m-k} \sp e^K$ and we will ignore the normalising constants $a_k$ until the final step. Upon a first application of $\nM$ we obtain a section of $\T^*\Mone\otimes \F$
\begin{align*}
\nM u_K (\dss)^{m-k} \sp e^K &=	\sds u_K \dss \otimes (\dss)^{m-k}\sp e^K \\
						&\quad	+ u_K \dss \otimes \nM_{\sds} \left( (\dss)^{m-k}\sp e^K \right) \\
						&\quad +\tsum_{i} e_i u_K e^i \otimes (\dss)^{m-k} \sp e^K \\
						&\quad + \tsum_{i} u_K e^i \otimes \nM_{e_i}  \left( (\dss)^{m-k}\sp e^K \right) .
\end{align*}
Using \eqref{eqn:derivationbasis:tau} and \eqref{eqn:derivationbasis:X} to develop the terms involving $\nM_{\sds}$ and $\nM_{e_i}$ we group the result in terms of symmetric powers of $\dss$. In order to handle the equations we write
\begin{align}\label{eqn:4terms}
\nM u_K (\dss)^{m-k} \sp e^K=\boxed{1}+\boxed{2}+\boxed{3}+\boxed{4}
\end{align}
where
\begin{align*}
\boxed{1} &=	- (m-k) \tsum_{i,j} u_K g_{ij} e^i \otimes (\dss)^{m-k-1}\sp e^\co{}{j}K, \\
\boxed{2} &=  (\sds-m)u_K \dss \otimes (\dss)^{m-k}\sp e^K,	 \\
\boxed{3} &= \tsum_i e_i u_K e^i \otimes (\dss)^{m-k} \sp e^K -\sum_{i,j}u_K e^i\otimes (\dss)^{m-k}\sp\textstyle\sum_{k_r\in K} \chris{ij}{k_r} e^\co{k_r}{j}{K}, \\
\boxed{4} &= - \tsum_iu_K e^i \otimes (\dss)^{m-k+1} \sp \textstyle\sum_{k_r\in K} \delta_i^{k_r} e^\co{k_r}{}{K}.
\end{align*}
Taking the second derivative, we calculate at a point about which $\{e_i\}$ are normal coordinates. Of course, we only need to keep track of terms which are not subsequently killed upon applying the trace $\trETA$ (which, as the notation suggests is the trace map from $\T^*\Mone\otimes \T^*\Mone\to\R$ built using the metric $\eta$).

$\boxed{1}$. Considering the first term in \eqref{eqn:4terms}, applying $\nM_{\sds}$ provides only terms in the kernel of $\trETA$ and applying $\nM_{e_i}$ gives
\begin{align*}
\tsum_\ell e^\ell\otimes \nM_{e_\ell} \boxed{1} & = 
							-  (m-k) \tsum_{i,j,\ell} e_\ell u_K g_{ij} e^\ell\otimes e^i \otimes (\dss)^{m-k-1}\sp e^\co{}{j}K  \\
			&	\quad			-  (m-k) \tsum_{i,j,\ell} u_K g_{ij} e^\ell\otimes e^i \otimes 
								\nM_{e_\ell} \left(
												(\dss)^{m-k-1}\sp e^\co{}{j}{K}
											\right)
							+ \ker \trETA 
\end{align*}
and we immediately apply $\trETA$ to get
\begin{align*}
-\s^2(\trETA \circ \nM) \boxed{1}
&=
				(m-k)(\dss)^{m-k-1}\sp \left(
					\tsum_i e_i u_Ke^\co{}{i}K 
				\right)
				\\
&\quad		+(m-k) u_K \tsum_j \nM_{e_j} \left(
						(\dss)^{m-k-1}\sp e^\co{}{j}{K}
					\right).
\end{align*}
The first term of the preceding display reduces to the symmetric differential $(m-k)(\dss)^{m-k-1}\sp \d(u_K e^K)$ by Lemma~\ref{lem:simpleformulae}. The second term of the preceding display is calculated with the aid of \eqref{eqn:derivationbasis:X} and remembering that the connection coefficients cancel at the point of interest. Specifically
\begin{align*}
\nM_{e_j} \left(
	(\dss)^{m-k-1}\sp e^\co{}{j}{K}
\right)
\,=\,
(\dss)^{m-k-2}&\sp\left(
					-\tsum_i (m-k-1) g_{ij} e^\ct{}{i}{}{j}K
		\right)+\\
(\dss)^{m-k\phantom{+0}}&\sp\left(
					\delta_j^j e^K + \textstyle\sum_{k_r\in K} \delta_j^{k_r} e^\ct{}{j}{k_r}{}{K}
		\right).
\end{align*}
Observe that $\tsum_j \textstyle\sum_{k_r\in K} \delta_j^{k_r} e^\ct{}{j}{k_r}{}{K} = k e^K$. Using Lemma~\ref{lem:simpleformulae} again this time to recover $\adjtrace$, the result is
\begin{align*}
(\dss)^{m-k-2}&\sp\left(
					-(m-k)(m-k-1)\adjtrace
		\right)u_Ke^K +\\
(\dss)^{m-k-1}&\sp\left(
					(m-k)\d
		\right)u_Ke^K +\\
-\s^2(\trETA \circ \nM) \boxed{1} \,=\,
(\dss)^{m-k\phantom{+0}}&\sp\left(
					-(m-k)(n+1+k)
		\right)u_Ke^K.
\end{align*}

$\boxed{2}$. Considering the second term in \eqref{eqn:4terms} is much simpler. A second application of $\nM$ provides
\begin{align*}
\nM  \boxed{2} & = (\sds-m-1)(\sds-m) )u_K \dss\otimes \dss \otimes (\dss)^{m-k}\sp e^K \\
			&	\quad			 
					-(\sds-m)u_K \sum_{i,j} g_{ij} e^i\otimes e^j \otimes (\dss)^{m-k}\sp e^K + \ker \trETA.
\end{align*}
and the desired result is
\begin{align*}
-\s^2(\trETA \circ \nM) \boxed{2} \,=\,
(\dss)^{m-k}&\sp\left(
					(\sds-m+n)(\sds-m)
		\right)u_Ke^K,
\end{align*}

$\boxed{3}$. Considering the third term in \eqref{eqn:4terms} is somewhat similar to the first term in that $\nM_{\sds}$ provides only terms in the kernel of $\trETA$. Remembering that at the point of interest, the connection coefficients vanish, applying $\nM_{e_j}$ gives
\begin{align*}
\tsum_j e^j\otimes \nM_{e_j} \boxed{3} & = 
							\tsum_{i,j} e_je_i u_K e^j\otimes e^i \otimes (\dss)^{m-k} \sp e^K	  \\
			&	\quad		+ \tsum_{i,j} e_i u_K e^j\otimes e^i \otimes \nM_{e_j}\left((\dss)^{m-k} \sp e^K	\right) \\
			&	\quad		 -\tsum_{i,j,\ell}u_K e^\ell\otimes e^i\otimes (\dss)^{m-k}\sp\textstyle\sum_{k_r\in K} \left( \n_{e_\ell} \chris{ij}{k_r} \right) e^\co{k_r}{j}{K}
							+ \ker \trETA 
\end{align*}
and we immediately apply $\trETA$ to recover the rough Laplacian from the first and third terms in the previous display
\begin{align*}
-\s^2(\trETA \circ \nM) \boxed{3}
&=
				(\dss)^{m-k}\sp 
					\n^*\n (u_Ke^K)
				\\
&\quad		- \tsum_{i,j} g^{ij} e_i u_K  \nM_{e_j}\left((\dss)^{m-k} \sp e^K	\right)
\end{align*}
while the second term in the previous display is first treated using \eqref{eqn:derivationbasis:X} and then Lemma~\ref{lem:simpleformulae} to recover the symmetric differential and the divergence
\begin{align*}
- \tsum_{i,j} g^{ij} e_i u_K  \nM_{e_j}\left((\dss)^{m-k} \sp e^K	\right)
&=
				\tsum_{i,j,\ell} g^{ij} e_i u_K 		(m-k)(\dss)^{m-k-1}\sp g_{\ell j} e^\co{}{\ell}{K} \\
& \quad			+\tsum_{i,j,\ell} g^{ij} e_i u_K  (\dss)^{m-k+1}\sp\sum_{k_r\in K}\delta_{j}^{k_r} e^\co{k_r}{}{K}
				\\
&=				(m-k)(\dss)^{m-k-1}\sp  \d (u_Ke^K) \\
& \quad			-(\dss)^{m-k+1}\sp \div (u_Ke^K).
\end{align*}
The result is
\begin{align*}
(\dss)^{m-k-1}&\sp\left(
					(m-k)\d
		\right)u_Ke^K +\\
-\s^2(\trETA \circ \nM) \boxed{3} \,=\,
(\dss)^{m-k\phantom{+0}}&\sp\left(
					\n^*\n
		\right)u_Ke^K +\\
(\dss)^{m-k+1}&\sp\left(
					- \div 
		\right)u_Ke^K.
\end{align*}

$\boxed{4}$. Considering finally the fourth term in \eqref{eqn:4terms} we immediately remove the sum over $i$ using the Kronecker delta. Again $\nM_{\sds}$ provides only terms in the kernel of $\trETA$ and applying $\nM_{e_i}$ gives
\begin{align*}
\tsum_i e^i\otimes \nM_{e_i} \boxed{4} & = 
							-\tsum_{i}\tsum_{k_r\in K} e_i u_K e^i\otimes e^{k_r} \otimes (\dss)^{m-k+1}\sp e^\co{k_r}{}{K} \\
&\quad				- \tsum_i\tsum_{k_r\in K}	u_K e^i\otimes e^{k_r} \otimes	\nM_{e_i} \left( (\dss)^{m-k+1}\sp e^\co{k_r}{}{K} \right)			
					+ \ker \trETA.
\end{align*}
and we immediately apply $\trETA$ to get
\begin{align*}
-\s^2(\trETA \circ \nM) \boxed{4}
&=
				(\dss)^{m-k+1}\sp \left(
					\tsum_i \sum_{k_r\in K} g^{i k_r}e_i u_Ke^\co{k_r}{}K 
				\right)
				\\
&\quad		+\tsum_i\tsum_{k_r\in K}	g^{ik_r}u_K\nM_{e_i}  \left( (\dss)^{m-k+1}\sp e^\co{k_r}{}{K} \right)
\end{align*}
The first term provides the divergence $-(\dss)^{m-k+1}\sp  \div (u_Ke^K)$ while the second term is treated using \eqref{eqn:derivationbasis:X} and then Lemma~\ref{lem:simpleformulae} to recover a multiple of $u_Ke^K$ and a term involving $\trace$
\begin{align*}
\tsum_i\tsum_{k_r\in K}	g^{ik_r}u_K\nM_{e_i}  &\left( (\dss)^{m-k+1}\sp e^\co{k_r}{}{K} \right) \\
&=
				-(m-k+1)(\dss)^{m-k}\sp \tsum_{i,j}\tsum_{k_r\in K} g^{ik_r}g_{ij} e^\ct{}{j}{k_r}{}{K} \\
&\quad			- (\dss)^{m-k+2} \sp\tsum_{k_r\in K} \tsum_{k_p\in \co{k_r}{}{K}} g^{k_rk_p} e^\ct{k_p}{}{k_r}{}{K} \\
&=
				-k (m-k+1)(\dss)^{m-k}\sp u_Ke^K \\
&\quad			- (\dss)^{m-k+2}\sp \trace (u_Ke^K).
\end{align*}
The result is
\begin{align*}
-\s^2(\trETA \circ \nM) \boxed{4} \,=\,
(\dss)^{m-k\phantom{+0}}&\sp\left(
					-k(m-k+1)
		\right)u_Ke^K +\\
(\dss)^{m-k+1}&\sp\left(
					- \div 
		\right)u_Ke^K +\\
(\dss)^{m-k+2}&\sp\left(
					-\trace
		\right)u_Ke^K.
\end{align*}

Upon summation of these four terms coming from \eqref{eqn:4terms} we obtain
\begin{align*}
(\dss)^{m-k-2}&\sp\left(
					-(m-k)(m-k-1) \adjtrace
		\right)u^{(k)} +\\
(\dss)^{m-k-1}&\sp\left(
					2(m-k)\d
		\right)u^{(k)} +\\
\s^2 \,\nM^*\nM  (\dss)^{m-k} \sp  u^{(k)} \,=\,
(\dss)^{m-k\phantom{+0}}&\sp\left(
					\n^*\n + (\sds+\tfrac n2 - m)^2 - \tilde c_k
		\right)u^{(k)} +\\
(\dss)^{m-k+1}&\sp\left(
					-2k  \div 
		\right)u^{(k)} +\\
(\dss)^{m-k+2}&\sp\left(
					-k(k-1)\trace
		\right)u^{(k)}
\end{align*}
with constant $\tilde c_k$ as announced in the proposition. The final step is to reintroduce the normalisation constants $a_k$. Treating, for example, the term containing $(\dss)^{m-k-1}$ amounts to observing
\[
a_{k+1}^{-1} (m-k) a_k = \sqrt{(m-k)}
\]
This completes the demonstration.
\end{proof}

\begin{lemma}\label{lem:decompk:curvature}
For $u\in C^\infty(\Mone; \F)$ decomposed relative to the Minkowski scale \eqref{eqn:decomp:mink}, the curvature endomorphism acts diagonally with respect to the Minkowski scale and is given by
\[
\s^2 q(\MRiemann) u^{(k)} = \left(
							q(\Riemann) + k(n+k-1) - \adjtrace\trace
						\right)u^{(k)}.
\]
\end{lemma}
\begin{proof}
We need only concern ourselves with the effect of $q(\MRiemann)$ on $(\dss)^{m-k}\sp e^K$. It is easy to see from Lemma~\ref{lem:coneconnection} that ${\MRiemann}_{\sds, e_i}$ is the zero endomorphism, that ${\MRiemann}_{e_i, e_j} \dss=0$, and that $\eta( {\MRiemann}_{e_i, e_j} e_k^*, \dss)=0$. Therefore we need only calculate the effect of $q(\MRiemann)$ on $e^K$. The non-trivial information of $\MRiemann$ is encoded in the following equation:
\[
\tensor{{\MRiemann}}{_{ij}^{k}_{\ell}} = g_{j\ell} \delta_i^k - g_{i\ell} \delta_j^k + \tensor{\Riemann}{_{ij}^k_\ell}.
\]
We extend ${\MRiemann}_{e_i,e_j}$ to $\Ek$ giving
\[
{\MRiemann}_{e_i,e_j} e^K = \Riemann_{e_i,e_j} e^K + \sum_{k_r\in K} 
		\left( \delta_j^{k_r}g_{i\ell} - \delta_i^{k_r} g_{j\ell}\right)	
		e^\co{k_r}{\ell}{K}
\]
Calculating the interior product requires the metric, in particular
\[
\s^2 \hookprodETA{e^i}{  {\MRiemann}_{e_i,e_j} } = \hookprod{e^i}{ {\MRiemann}_{e_i,e_j} }
\]
where $\hookprodETA{}{}$ uses the metric $\eta$ to identify $\T\Mone$ with $\T^*\Mone$. Consequently calculating 
\[
\sum_i \left(
\s^2 \hookprodETA{e^i}{  {\MRiemann}_{e_i,e_j}   e^K }
-
\hookprod{e^i}{ \Riemann_{e_i,e_j} e^K }
\right)
\]
gives
\[
\sum_{i}
\sum_{k_r\in K} 
		\left( \delta_j^{k_r}g_{i\ell} - \delta_i^{k_r} g_{j\ell}\right)
			\left(
					g^{i\ell} e^\co{k_r}{}{K}
					+
					 \sum_{k_{p} \in \co{k_r}{}{K}} g^{ik_p} e^\ct{k_p}{}{k_r}{\ell}{K}
			\right).
\]
Applying $\sum_{j} e^j\sp$ to the preceding display provides $\s^2 q(\MRiemann)- q(\Riemann)$. Splitting the calculation into four terms, the results are
\begin{align*}
\tsum_{i,j}\tsum_{k_r\in K} e^j \sp \delta_j^{k_r}g_{i\ell}	
			g^{i\ell} e^\co{k_r}{}{K} 
			&= k(n+1) e^K, \\
-\tsum_{i,j}\tsum_{k_r\in K} e^j \sp \delta_i^{k_r} g_{j\ell} 
			g^{i\ell} e^\co{k_r}{}{K} 
			&= -k e^K,\\
\tsum_{i,j}\tsum_{k_r\in K}  \tsum_{k_{p} \in \co{k_r}{}{K}} e^j \sp  \delta_j^{k_r}g_{i\ell} 
			 g^{ik_p} e^\ct{k_p}{}{k_r}{\ell}{K}
			& = k(k-1) e^K, \\
-\tsum_{i,j}\tsum_{k_r\in K}\sum_{k_{p} \in \co{k_r}{}{K}} e^j \sp  \delta_i^{k_r} g_{j\ell} 
			  g^{ik_p} e^\ct{k_p}{}{k_r}{\ell}{K}
			 &=-\adjtrace\trace e^K.
\end{align*}
Upon summation of these four terms, the proof is complete.
\end{proof}

\begin{proposition}\label{prop:decompQ:k:mink:tracefree}
Suppose $u\in C^\infty(\Mone; \F)$, decomposed relative to the Minkowski scale \eqref{eqn:decomp:mink}, is trace-free with respect to the trace operator $\traceSETA$. Then the conjugated d'Alembertian $\QVasyUp$ is given by
\begin{align*}
a_{k+1}\,(\dss)^{m-k-1}&\sp\big(
					2b_k\d
		\big)u^{(k)} +\\
 \QVasyUp   a_k\, (\dss)^{m-k} \sp  u^{(k)} \,=\,
a_{k\phantom{+0}}\,(\dss)^{m-k\phantom{+0}}&\sp\big(
					\Delta + (\sds)^2 -  c_k'\,
		\big)u^{(k)} +\\
a_{k-1}\,(\dss)^{m-k+1}&\sp\big(
					-2b_{k-1}  \div 
		\big)u^{(k)}
\end{align*}
with constants $a_k$, $b_k$ announced in Proposition~\ref{prop:decompQ:k:mink} and the modified constants
\[
c_k' = c_k - (m-k)(m-k-1).
\]
\end{proposition}
\begin{proof}
This follows directly from the structure of $\QVasyUp$ given in Proposition~\ref{prop:decompQ:k:mink} and the condition that $\trace u^{(k)} = - b_{k-2}b_{k-1} u^{(k-2)}$ coming from Lemma~\ref{lem:trace:FE}.
\end{proof}

%%%%% %%%%% %%%%% %%%%% %%%%% %%%%% %%%%% %%%%% %%%%% %%%%% %%%%% %%%%%

\subsection{The indicial family of $\QVasyUp$}

\begin{definition}\label{defn:QQ}
Denote by $\QVasyDown$ the indicial family of the operator $\QVasyUp \in \Diff^2_\bbb(\overline{\Mone} ; \F )$ relative to the Minkowski scale $\s$.
\[
\QVasyDown = \IF_\s(\QVasyUp;\lambda) \in \Diff^2(\Xone; \E).
\]
\end{definition}
The previous section introduced $\QVasyUp$ as a differential operator on $\F$ above $\Mone$ however, from the structure of $\QVasyUp$ given as announced in Proposition~\ref{prop:decompQ:k:mink}, it is clear that the operator extends to $\overline\Mone$. Moreover by the same proposition we immediately get the structure of $\QVasyDown$.
\begin{proposition}\label{prop:decompQQ:k:mink}
For $u=\sum_{k=0}^m u^{(k)}\in C^\infty(\Xone; \E)$, the operator $\QVasyDown$ is given by
\begin{align*}
&\big(
					-b_k b_{k+1} \adjtrace
		\big)u^{(k)} +\\
&\big(
					2 b_k \d
		\big)u^{(k)} +\\
\QVasyDOWNL  u^{(k)} \,=\,
&\big(
					\Delta +\lambda^2 -  c_k -  \adjtrace\trace
		\big)u^{(k)} +\\
&\big(
					-2b_{k-1}  \div 
		\big)u^{(k)} +\\
&\big(
					-b_{k-2}b_{k-1}\trace
		\big)u^{(k)}
\end{align*}
with constants
\[
b_k = \sqrt{m-k},
\qquad
c_k = \tfrac{n^2}{4} + m(n+2k+1) - k(2n+3k-1).
\]
Consequently, there exist $\DVasyDown\in \Diff^1(\Xone;\End\E)$ and $\ZeroOrderVasyDown\in C^\infty(\Xone;\End\E)$ independent of $\lambda$ such that
\[
\QVasyDOWNL = \n^*\n + \lambda^2 + \DVasyDown + \ZeroOrderVasyDown
\]
\end{proposition}
\begin{proposition}\label{prop:QQ:selfadjoint}
The family of differential operators $\QVasyDown$ is, upon restriction to $\lambda\in i\R$, a family of formally self-adjoint operators with respect to the inner product
\[
\LinnprodSS{u}{v} = \int_{\Xone} \sum_{k=0}^m (-1)^{m-k} \innprod{ u^{(k)} }{ v^{(k)} } \, d\vol_g
\]
where $u={\sum_{k=0}^m u^{(k)}}$, $v=\sum_{k=0}^m v^{(k)}$ for $u^{(k)},v^{(k)} \in C^\infty_c (\Xone; \Ek)$. Moreover, for all $\lambda$, $\QVasyDOWNL^* = \QVasyDown_{-\bar\lambda}$.
\end{proposition}
\begin{proof}
Lemmas~\ref{lem:selfadjointpreservation} and~\ref{lem:Qselfadjoint}.
\end{proof}

The operator $\QVasyUp$ preserves the subbundle $\F\cap\ker\traceSETA$ by Lemma~\ref{lem:Qcommutestrace}. As $\pistarS$ is algebraic, we may consider it as a map from $\E$ over $\Xone$ to $\F$ over $\Mone$. We thus obtain the subbundle $\E\cap\ker(\traceSETA\circ\pistarS)$ over $\Xone$; symmetric tensors above $\Xone$ which are trace-free with respect to the ambient trace operator $\traceSETA$. It thus follows that $\QVasyDown$ may also be considered a family of differential operators on this subbundle and we obtain
\begin{proposition}\label{prop:decompQQ:k:mink:tracefree}
For $u=\sum_{k=0}^m u^{(k)}\in C^\infty(\Xone; \E)\cap\ker(\traceSETA\circ\pistarS)$, the operator $\QVasyDown$ is given by
\begin{align*}
&\big(
					2 b_k \d
		\big)u^{(k)} +\\
\QVasyDOWNL  u^{(k)} \,=\,
&\big(
					\Delta +\lambda^2 -  c_k'
		\big)u^{(k)} +\\
&\big(
					-2b_{k-1}  \div 
		\big)u^{(k)}.
\end{align*} 
\end{proposition}

%%%%% %%%%% %%%%% %%%%% %%%%% %%%%% %%%%% %%%%% %%%%% %%%%% %%%%% %%%%%

%%%%% %%%%% %%%%% %%%%% %%%%% %%%%% %%%%% %%%%% %%%%% %%%%% %%%%% %%%%%
%%%%% %%%%% %%%%% %%%%% %%%%% %%%%% %%%%% %%%%% %%%%% %%%%% %%%%% %%%%%
%%%%% %%%%% %%%%% %%%%% %%%%% %%%%% %%%%% %%%%% %%%%% %%%%% %%%%% %%%%%

\section{The operator $\PVasyUp$ and its indicial family}\label{sec:P}

This section introduces the operator $\PVasyUp$ on $\Mthree$ and its indicial family $\PVasyDown$ on $\Xthree$ and similar results to those presented for $\QVasyUp$ and $\QVasyDown$ are given. The relationship between these two constructions is also detailed.

\subsection{Euclidean scale}

The manifold $\Mthree=\RplusT\times \Xthree$ has been equipped with the Lorentzian metric $\eta$ which agrees with the Lorentzian cone metric put on $\Mone$. Recalling the smooth chart $U=(0,1)_\mu\times Y\subset \Xone\subset\Xthree$ the metric, on $\RplusT\times U$ takes the form of \eqref{eqn:rho-2eta} and we may assume that this is the form of $\eta$ on the larger chart $\RplusT\times U^2$ where $U^2=(-1,1)_\mu\times Y$. For later use we record the behaviour of $\nMthree$.
\begin{lemma}\label{lem:Mthreeconnection}
On the chart $\RplusT\times(-1,1)_\mu\times Y$ with $\{e_i\}_{i=1}^n$ a local holonomic frame on $Y$ with dual frame $\{e^i\}_{i=1}^n$ such that $h=\sum_{i,j} h_{ij} e^i\otimes e^j$, the connection $\nMthree$ acts in the following manner:
\begin{align*}
\nMthree_{\tdt} \dtt &= 0,		
					&\nMthree_{\partial_\mu}\dtt&=0, \\
\nMthree_{\tdt} d\mu &= -d\mu,		
					&\nMthree_{\partial_\mu}d\mu&=-\dtt, \\
\nMthree_{\tdt} e^i &= -e^i,		
					&\nMthree_{\partial_\mu}e^i&=-\tfrac 12 h^{ij}(\partial_\mu h_{jk}) e^k,
\end{align*}
and
\begin{align*}
\nMthree_{e_i}\dtt&= -(\partial_\mu h_{ij}) e^j,\\
\nMthree_{e_i}d\mu&= -2((1-\mu\partial_\mu)h_{ij}) e^j,\\
\nMthree_{e_i}e^j&= -\delta_i^j\dtt -\tfrac 12 h^{jk}(\partial_\mu h_{ik}) d\mu + \nY_{e_i} e^j.
\end{align*}
\end{lemma}
Motivated by the structure of $\QVasyUp$ from the previous section we define the second of a our two main operators
\begin{definition}
The second-order differential operator $\PVasyUp\in\Diff^2(\Mthree; \F)$ is the following conjugation of the d'Alembertian:
\[
\PVasyUp : \left\{	 \begin{array}{rcl}
			C^\infty(\Mone ; \F) & \to & C^\infty(\Mthree ; \F) \\
			u & \mapsto &  \t^{\frac n2-m+2} \DeltaAmb \t^{-\frac n2+m} \, u
		\end{array}
		\right.
\]
\end{definition}

Note that on $\Mone\subset\Mthree$ there is a trivial correspondence between $\PVasyUp$ and $\QVasyUp$,
\[
\PVasyUp = \rho^{-\frac n2+m-2} \QVasyUp \rho^{\frac n2-m}
\]
and that, since $\rho=1$ on $\Xone\backslash U$, we have equality $\PVasyUp=\QVasyUp$ on $\Mone\backslash (\R^+\times U)$.

\begin{lemma}
The operator $\PVasyUp\in \Diff^2(\Mthree;\F)$ naturally extends to an operator $\PVasyUp\in\Diff^2_\bbb(\overline\Mthree;\bF)$ and is b-trivial.
\end{lemma}
\begin{proof}
The important point is to verify that at $\mu=0$, $\PVasyUp$ fits into the b-calculus framework. This is reasonably clear from Lemma~\ref{lem:Mthreeconnection}. Indeed, the Lie algebra of b-vector fields is generated by $\{\tdt,\partial_\mu, e_i\}$ where $\{e_i\}_{i=1}^n$ is a local holonomic frame on $Y$, while the b-cotangent bundle has basis $\{\dtt,d\mu, e^i\}$ with $\{e^i\}_{i=1}^n$ the dual frame on $\T^*Y$. Lemma~\ref{lem:Mthreeconnection} thus shows that $\nMthree$ is a b-connection. Taking the trace using $\eta$ and then multiplying by $\t^2$ is equivalent to taking the trace with $\t^{-2}\eta$ whose structure \eqref{eqn:rho-2eta} indicates it is a b-metric. Therefore $\t^{2}\,\nMthree^*\nMthree$ is a b-differential operator. That $\t^2\,\nMthree^*\nMthree$ is b-trivial is also immediate from Lemma~\ref{lem:Mthreeconnection} and the structure of $\t^{-2}\eta$. A similar line of reasoning for $q(\MthreeRiemann)$ (which uses one application of the inverse of the metric $\eta$) shows that $\t^2 \DeltaAmb$ is also a b-differential operator. The final conjugation by powers of $\t$ preserves the b-structure (and its b-triviality) as it merely conjugates appearances of $\tdt$. This implies the result.
\end{proof}

\begin{lemma}\label{lem:Pselfadjoint}
The differential operator $\PVasyUp$ is formally self-adjoint with respect to the inner product
\[
\LinnprodTETA{ u}{v} = \int_{\Mthree} \innprodTETA{  u }{ v } \, \dtt    \meas,
\qquad
u,v\in C^\infty_c (\Mthree; \F).
\]
\end{lemma}
\begin{proof}
By the correspondence between $\PVasyUp$ and $\QVasyUp$ on $\Mone\backslash(\R^+\times U)$ and Lemma~\ref{lem:Qselfadjoint}, it suffices to verify this claim when $u,v$ are supported on $\RplusT\times U^2$.
The d'Alembertian is self-adjoint with respect to the following inner product
\[
\LinnprodETA{ u}{ v } = \int_{\Mthree} \innprodETA{  u }{ v } \, d\vol_\eta,
\qquad
u,v\in C^\infty_c (\RplusT\times U^2; \F).
\]
The two inner products on the fibres of $\F$ are related via the Euclidean scale analog of \eqref{eqn:twoinnprods:eta}. Tracking the effects of the conjugations by powers of $\t$ on $\DeltaAmb$ as well as the multiplication by $\t^2$ in order to obtain $\PVasyUp$ implies self-adjointness when using the inner product $\innprodTETA{\cdot}{\cdot}$ with the measure $\t^{-n-2}d\vol_\eta$. As $\det \eta = - \tfrac14 \t^{2n+2} \det h$ we have
\[
\t^{-n-2}d\vol_\eta = \tfrac12  \dtt  d\mu \, d\vol_h. \qedhere
\]
\end{proof}

\subsection{The indicial family of $\PVasyUp$}\label{subsec:indicialfamilyP}

\begin{definition}
Denote by $\PVasyDown$ the indicial family of the operator $\PVasyUp \in \Diff^2_\bbb(\overline{\Mthree} ; \bF )$ relative to the Euclidean scale $\t$.
\[
\PVasyDOWNL = \IF_\t(\PVasyUp;\lambda) \in \Diff^2(\Xthree; \E).
\]
\end{definition}
Lemma~\ref{lem:indicialfamily:changeofscales} gives the following proposition (whose final statement follows as $\rho$ is constant on $\Xone\backslash U$).
\begin{proposition}\label{prop:PPandQQ:changeofscales}
On $\Xone\subset\Xthree$ the indicial family operators $\PVasyDown$ and $\QVasyDown$ are related by
\[
\PVasyDOWNL = \rho^{-\lambda - \frac n2 + m -2} J \QVasyDOWNL J^{-1} \rlnm
\]
with $J$ presented in Lemma~\ref{lem:innprod:changeofscales}. Moreover, on $\Xone\backslash U$, we have equality $\PVasyDown=\QVasyDown$.
\end{proposition}
\begin{proposition}\label{prop:PPselfadjoint}
The family of differential operators $\PVasyDown$ is, upon restriction to $\lambda\in i\R$, a family of formally self-adjoint operators with respect to the inner product
\[
\LinnprodTT{u}{v} = \int_{\Xthree} \innprodTT{ u }{ v } \, \meas,
\qquad
u,v\in C^\infty_c(\Xthree ; \E).
\]
Moreover, for all $\lambda$, $\PVasyDOWNL^* = \PVasyDown_{-\bar\lambda}$.
\end{proposition}

%%%%% %%%%% %%%%% %%%%% %%%%% %%%%% %%%%% %%%%% %%%%% %%%%% %%%%% %%%%%
%%%%% %%%%% %%%%% %%%%% %%%%% %%%%% %%%%% %%%%% %%%%% %%%%% %%%%% %%%%%
%%%%% %%%%% %%%%% %%%%% %%%%% %%%%% %%%%% %%%%% %%%%% %%%%% %%%%% %%%%%

\section{Microlocal Analysis}\label{sec:analysis}

This section constructs an inverse to the family $\PVasyDown$ introduced in the preceding section. This is done by first showing that the family is a family of Fredholm operators and then by considering a Cauchy problem which provides an inverse for $\Re\lambda\gg 1$. In \cite{v:ml:functions,zworski:vm}, the procedure is described for functions, rather than symmetric tensors. We are required to alter only minor details in order to apply the technique to symmetric tensors.

\subsection{Function spaces}\label{subsec:functionspaces}

From Subsection~\ref{subsec:vecbundles}, we have the space of $\Lsections$ sections $\LsectionsT(\Xthree;\E)$. This defines $H^s_\loc(\Xthree; \E)$, the space of (locally) $H^s$ sections for $s\in\R$. For all notions of Sobolev regularity, we will only use the Euclidean scale; we thus need not decorate these spaces with a subscript $\t$.

As is standard, we denote by $\dot C^\infty(\Xtwo ; \E)$ the set of smooth sections which are extensible to smooth sections over $\Xthree$ and whose support is contained in $\overline\Xtwo$. 
And by $C^\infty(\overline\Xtwo ; \E)$ all smooth sections which are smoothly extensible to $\Xthree$.

Following \cite[Appendix B.2]{hormander3} we obtain, for $s\in\R$, the Sobolev spaces
\[
\dot H^s(\overline \Xtwo ; \E )
\quad
\textrm{and}
\quad
\overlineH^s(\Xtwo; \E)
\] 
which are, respectively, the set of elements in $H^s_\loc(\Xthree; \E)$ supported by $\overline\Xtwo$ and the space of restrictions to $\Xtwo$ of $H^s_\loc(\Xthree; \E)$. Then $\dot H^s(\overline \Xtwo ; \E )$ gets its norm directly from that of $H^s_\loc(\Xthree; \E)$ while the norm of an element in $\overlineH^s(\Xtwo; \E)$ is that obtained by taking the infimum of the norms of all permissible extensions of the element which have compact support in $\Xthree$. (Such norms will be denoted, for simplicity, by $\norm{\cdot}{\dot H^s}$ and $\norm{\cdot}{\overlineH^s}$. Furthermore, if an object is supported away from $\CS$, these norms correspond and we may simply write $\norm{\cdot}{H^s}$.)

The inner product $\innprodTT{\cdot}{\cdot}$ gives the $\Lsections$ pairing
\[
\LinnprodTT{\cdot}{\cdot}
		 : \dot C^\infty(\Xtwo ; \E) \times C^\infty(\overline\Xtwo ; \E) \to \C
\]
which extends by density \cite[Theorem B.2.1]{hormander3} to a pairing between the spaces $\dot H^{-s}(\overline\Xtwo ; \E )$ and $\overlineH^s(\Xtwo ; \E )$ providing the identification of dual spaces
\begin{align}\label{eqn:dualHsdotbar}
(\overlineH^s(\Xtwo ; \E ))^* \simeq \dot H^{-s}(\overline \Xtwo ; \E ),
\qquad
s\in\R.
\end{align}

\begin{definition}
For $s\in\R$, let $\XVasyML^s$ and $\YVasyML^s$ be the following two spaces
\begin{align*}
\YVasyML^s	& = \overlineH^s (\Xtwo ; \E), \\
\XVasyML^s	& = \{ u : u\in \YVasyML^s, \PVasyDown u \in \YVasyML^{s-1} \}
\end{align*}
These spaces come with the standard norms, in particular,
\[
\| u \|_{\XVasyML^{s}} = \| u \|_{\YVasyML^{s}} + \|\PVasyDown  u \|_{\YVasyML^{s-1}},
\qquad
u\in \XVasyML^s.
\]
\end{definition}
\begin{remark}
It will be seen that $\lambda$ does not appear in the principal symbol of $\PVasyDown$, it is thus unimportant to explicit with respect to what value of $\lambda$ the preceding norm is taken as all such norms are equivalent.
\end{remark}

When restricting to $U^2\subset\Xthree$ we will let $\{e_i\}_{i=1}^n$ denote an orthonormal frame for $(Y,h)$ (which depends on $\mu\in (-1,1)$) and by $\{e^i\}_{i=1}^n$ its dual frame. The frames are completed to frames for $TU^2$ and $\T^*U^2$ by including $\partial_\mu$ and $d\mu$ respectively. A dual vector will take the notation
\begin{align}\label{eqn:localbasisT*U}
\xi d\mu + \sum_{i=0}^n \eta_i e^i \in \T^*U^2.
\end{align}

The following subsection proves the following two propositions.
\begin{proposition}\label{prop:fredholm}
For fixed $s$, the family of operators
\[
\PVasyDown : \XVasyML^{s} \to \YVasyML^{s-1}
\]
is Fredholm for $\Re \lambda > \frac12 - s$.
\end{proposition}
\begin{proof}
Lemmas~\ref{lem:Pfinitekernel} and~\ref{lem:Pfinitecokernel}.
\end{proof}
\begin{proposition}\label{prop:inverse:existence}
For fixed $s$, the Fredholm operator $\PVasyDOWNL: \XVasyML^s\to \YVasyML^{s-1}$ are Fredholm of index 0 for $\Re \lambda > m+\frac12 - s$ and it has a meromorphic inverse 
\[
\PVasyDown^{-1} : \YVasyML^{s-1} \to\XVasyML^s
\]
with poles of finite rank.
\end{proposition}
\begin{proof}
Lemmas~\ref{lem:Ptrivialkernel} and~\ref{lem:Ptrivialcokernel}.
\end{proof}

\subsection{Proofs of Propositions~\ref{prop:fredholm} and~\ref{prop:inverse:existence}}

On $\RplusT\times U^2$, the inverse of the metric $\eta$ takes the form
\begin{align}\label{eqn:structurerho2eta-1}
\t^2 \eta^{-1} = -2 \tdt \sp \partial_\mu + 2\mu \partial_\mu\sp\partial_\mu + h^{-1}
\end{align}
which implies to highest order for $\t^2\,\nMthree^*\nMthree$, that
\[
\t^2\,\nMthree^*\nMthree = -4\mu \partial_\mu^2 + 4\tdt \partial_\mu + \Delta_h + \Diff^1(\RplusT\times U^2 ; \End\F)
\]
where $\Delta_h$ may be considered the rough Laplacian on $(Y,h)$. Considering $\PVasyDown$, conjugation by $\t^{-\frac n2 + m}$ replaces $\tdt$ by $(\tdt - \tfrac n2 + m)$ and we can absorb the newly created term $4(-\tfrac n2 + m)\partial_\mu$ into $\Diff^1(\RplusT\times U^2 ; \End\F)$. Also, the curvature term is of order zero so
\[
\PVasyUp = -4\mu \partial_\mu ^2 + 4\tdt \partial_\mu + \Delta_h + \LowerOrderTermVasyUp
\]
for some $\LowerOrderTermVasyUp \in \Diff^1(\RplusT\times U^2 ; \End\F)$. This structure of $\PVasyUp$ immediately gives the structure of $\PVasyDown$ to highest order. Keeping track of the term $4\tdt\partial_\mu$ for the moment, we write
\begin{align}\label{eqn:PPl-using-LOT}
\PVasyDOWNL = -4\mu\partial_\mu^2 - 4\lambda \partial_\mu + \Delta_h + \LowerOrderTermVasyDOWNL.
\end{align}
where $\LowerOrderTermVasyDOWNL\in\Diff^1(U^2 ; \End\E)$ is the indicial family of $\LowerOrderTermVasyUp$. The most obvious conclusion we draw from such a presentation of $\PVasyDown$ is that $\PVasyDown$ is a family of elliptic operators on $U^2\cap\{\mu>0\}$ and a family of strictly hyperbolic operators for $\{\mu<0\}$ (with respect to the level sets $\{\mu=\mathrm{constant}\}$). Of course the ellipticity extends to all of $\Xone$. The principal symbol on $U^2$ is also immediately recognisable as
\[
\principalsymbol{\PVasyDown} = 4\mu \xi^2 + |\eta|^2
\]
using the notation from \eqref{eqn:localbasisT*U} and $|\eta|^2=\sum_{i=1}^n \eta_i^2$. And on $U^2$, the Hamiltonian vector field associated with $\principalsymbol{\PVasyDown}$ is
\[
\Hamiltonian{\principalsymbol{\PVasyDown}} = 8\mu \xi \partial_\mu - 4\xi^2 \partial_\xi + \Hamiltonian{|\eta|^2}.
\]

The strategy to obtain a Fredholm problem is to combine standard results for elliptic and hyperbolic operators with some analysis performed at the junction $Y=\{\mu=0\}$. The analysis was first presented in \cite[Section 4.4]{v:ml:inventiones}. It turns out the dynamics of interest are those of radial sources and sinks \cite[Definition E.52]{dyatlov-zworski:book}. The original radial estimates of Melrose \cite{melrose:scat} on asymptotically Euclidean spaces have been adapted to functions on asymptotically hyperbolic spaces by Vasy \cite{v:ml:inventiones}. Indeed, to see that such dynamics are relevant for $\PVasyDown$, consider $\principalsymbol{\PVasyDown}$ and $\Hamiltonian{\principalsymbol{\PVasyDown}}$ given in the preceding displays. Define the characteristic variety $\Sigma\subset \T^*\Xtwo\backslash 0$ which is contained in $\T^*U$. As $(\mu,y,0,\eta)\not\in\Sigma$, we may split $\Sigma=\Source\sqcup\Sink$ given by $\SourceSink=\Sigma\cap\{\pm\xi>0\}$. At $Y$ remark that
\[
\Sigma\cap \T^*_YU = \{ (0,y,\xi,0) : \xi\neq 0 \} \subset N^*Y
\]
and recalling the projection $\kappa:\T^*U\backslash0\to\partial\rct U$ define
\[
\cSource=\kappa(\Source\cap Y),
\quad
\cSink=\kappa(\Sink\cap Y).
\]
In \cite[Section 3.2]{v:ml:functions}, it is shown that $\cSourceSink$ are respectively a source and a sink for $\principalsymbol{\PVasyDown}$. In order to apply Lemmas~\ref{lem:highregularity} and~\ref{lem:lowregularity}, we introduce the principal symbol of the imaginary part of $\PVasyDown$. 
By Remark~\ref{rem:highlowregularity}, $\Hamiltonian{\principalsymbol{\PVasyDown}}=\Hamiltonian{\principalsymbol{\PVasyDown^*}}$ and by Proposition~\ref{prop:PPselfadjoint}, $\PVasyDOWNL^*=\PVasyDown_{-\bar\lambda}$ hence $\principalsymbol{\Im \PVasyDown}=-\principalsymbol{\Im \PVasyDown^*}$. Also, by a direct calculation using the structure of $ \Hamiltonian{\principalsymbol{\PVasyDown}}$,
\begin{align}\label{eqn:HamiltonianSourceSink}
\japbrak{\xi+\eta}^{-1} \Hamiltonian{\principalsymbol{\PVasyDown}}\log\japbrak{\xi+\eta} = \mp 4,
\quad
\textrm{on $\cSourceSink$.}
\end{align}
In fact Proposition~\ref{prop:PPselfadjoint} along with \eqref{eqn:PPl-using-LOT} gives more precisely
\[
\Im \PVasyDOWNL = \frac{ \PVasyDOWNL - \PVasyDOWNL^*}{2i} = 4i (\Re\lambda) \partial_\mu + \frac{ \LowerOrderTermVasyDown_{\lambda} - \LowerOrderTermVasyDown_{-\bar\lambda} }{2i}
\]
however as $\LowerOrderTermVasyUp$ is first order, $\LowerOrderTermVasyDOWNL$ may be written as the sum of a first order operator independent of $\lambda$ and a zeroth order operator (which may depend on $\lambda$). Therefore
\begin{align}\label{eqn:imaginary:subprincipalsymbol}
\principalsymbol{ \Im \PVasyDOWNL } = -4 \Re\lambda \, \xi.
\end{align}
Bringing this altogether in preparation for the proof of Proposition~\ref{prop:fredholm} we have
\begin{lemma}\label{lem:thresholdconditions}
For $\PVasyDown$, $\cSource$ is a source, while $\cSink$ is a source for $-\PVasyDown$. In both situations, the threshold condition, when working on $H^s(\Xtwo;\E)$, is satisfied if
\[
s > - \Re \lambda + \tfrac12.
\]
For $\PVasyDown^*$, $\cSink$ is a sink, while $\cSource$ is a sink for $-\PVasyDown^*$. In both situations, the threshold condition, when working on $H^{\tilde s}(\Xtwo;\E)$, is satisfied if
\[
\tilde s < \Re \lambda + \tfrac 12.
\]
\end{lemma}
\begin{proof}
We explain the first result, all others are similar after taking into account Remark~\ref{rem:highlowregularity}. On $\cSource$, by \eqref{eqn:HamiltonianSourceSink} and \eqref{eqn:imaginary:subprincipalsymbol},
\[
\japbrak{\xi+\eta}^{-1} ( \principalsymbol{\Im \PVasyDown} + (s-\tfrac 12) \Hamiltonian{\principalsymbol{\PVasyDown}}\log\japbrak{\xi+\eta})
=
-4 (\Re\lambda+s-\tfrac12).
\]
For this to be negative definite requires precisely that $s >- \Re \lambda + \tfrac12$.
\end{proof}

\begin{lemma}\label{lem:Pfinitekernel}
Restricting to $s>-\Re\lambda + \frac 12$, the operators $\PVasyDOWNL : \XVasyML^s\to \YVasyML^{s-1}$ have finite dimensional kernels.
\end{lemma}
\begin{proof}
It suffices to obtain an estimate, for $u\in\XVasyML^s$, of the form
\[
\norm{ u }{\overlineH^s} \le C \left(	\norm{ \PVasyDOWNL u }{\overlineH^{s-1}}	+ \norm{\psi u}{H^{-N}}		\right).
\]
for some $\psi$ supported on $\{\mu>-\frac12\}$ and such that $\psi=1$ near $\{\mu>-\frac12+\ve\}$. This is done by writing $u=(\psi_- + \psi_0 + \psi_+)u$ with the supports of $\psi_-, \psi_0, \psi_+$ respectively contained in $\{\mu<-\ve\}, \{|\mu|<2\ve\}, \{\mu>\ve\}$. 
The estimate for $\psi_+u$ is due to ellipticity of $\PVasyDown$. 
The estimate for $\psi_-u$ is due to hyperbolicity which allows us to reduce to the estimate for $\psi_0 u$:
\[
\norm{\psi_- u}{\overlineH^s} \le C \left( \norm{ \PVasyDOWNL u}{\overlineH^{s-1}} + \norm{ \psi_0 u}{H^{s}} \right).
\]
The estimate for $\psi_0 u$ is obtained by microlocalising. Away from $\Sigma$, ellipticity gives the result, while near $\Sigma$, propagation of singularities implies that the norms can be controlled by $\cSourceSink$. The high regularity results for $\cSource$ and $\cSink$ from Lemma~\ref{lem:highregularity} are applicable as these are sources for $\PVasyDown$ and $-\PVasyDown$ respectively. Lemma~\ref{lem:thresholdconditions} ensures that the threshold conditions are satisfied (by hypothesis of this proposition). The desired estimate is obtained.
\end{proof}

\begin{lemma}\label{lem:Pfinitecokernel}
Restricting to $s>-\Re\lambda + \frac 12$, the operators $\PVasyDOWNL : \XVasyML^s\to \YVasyML^{s-1}$ have finite dimensional cokernels.
\end{lemma}
\begin{proof}
To show that the range is of finite codimension we study the adjoint operator $\PVasyDown^*$. By \eqref{eqn:dualHsdotbar} the dual space of $\overlineH^{s-1}(\Xtwo; \E)$ is $\dot H^{1-s}(\overline\Xtwo;\E)$ and the dimension of the kernel of $\PVasyDown^*$ equals the dimension of the cokernel of $\PVasyDown$. It suffices to obtain an estimate of the form
\[
v\in \dot H^{1-s}(\overline\Xtwo;\E)\cap \ker \PVasyDown^*
\quad
\implies
\quad
\norm{ v }{\dot H^{1-s}} \le C  \norm{\psi v }{H^{-N}}		
\]
with $\psi$ as defined in the previous proof. Again, we use the partition $v=(\psi_- + \psi_0 + \psi_+)v$. 
Again, the estimate for $\psi_+v$ is due to ellipticity of $\PVasyDown^*$. This time, the estimates for $\psi_-v$ are immediate due to hyperbolicity and the requirement at $\CS$ that $v$ vanish to all orders which implies that $v=0$ on $\{\mu<0\}$.
The estimate for $\psi_0v$ is obtained by microlocalising. 
(Away from $\Char(P)$, the result is obtained by ellipticity.)
The low regularity results for $\cSink$ and $\cSource$ from Lemma~\ref{lem:lowregularity} are applicable as these are sinks for $\PVasyDown^*$ and $-\PVasyDown^*$ respectively. Lemma~\ref{lem:thresholdconditions} ensures that the threshold conditions are satisfied. Therefore there exists $A,B\in\Psi^0(\Xtwo)$ with $\Char(A)\cap\cSourceSink=\varnothing$ and $\WF(B)\cap\cSourceSink=\varnothing$ such that $\norm{ A\psi_0v}{H^{1-s}}\le C( \norm{B\psi_0v}{H^{1-s}} + \norm{\psi v}{H^{-N}})$.
As $v=0$ on $\{\mu<0\}$ and is smooth (by ellipticity of $\PVasyDown^*$) on $\{\mu>0\}$, we have $\WF(B\psi_0v)\cap \Char(\PVasyDown^*)=\varnothing$ so microellipticty gives $\norm{B\psi_0v}{H^{1-s}} \le C \norm{\psi v}{H^{-N}}$.
The desired estimate is obtained.
\end{proof}

\begin{lemma}\label{lem:Ptrivialkernel}
For $\PVasyDOWNL$ with $\lambda\in\R$ acting on $\overlineH^s(\Xtwo;\E)$, the kernel of $\PVasyDOWNL$ is trivial for $\lambda\gg1$.
\end{lemma}
\begin{proof}
Consider $u\in\ker\PVasyDOWNL$. By the estimate obtained in Lemma~\ref{lem:Pfinitekernel}, $u\in C^\infty(\overline\Xtwo;\E)$. Restricting our attention to $\{\mu>0\}$, Proposition~\ref{prop:PPandQQ:changeofscales} gives
\[
\rho^{-\lambda - \frac n2 + m -2} J \QVasyDOWNL J^{-1} \rlnm u =0
\]
so defining $\tilde u = J^{-1} \rlnm u$ we get $\QVasyDOWNL \tilde u =0$. Or by Proposition~\ref{prop:decompQQ:k:mink},
\[
( \n^*\n + \lambda^2 + \DVasyDown + \ZeroOrderVasyDown ) \tilde u = 0.
\]
Now $\DVasyDown$ may be bounded (up to a constant) by $\n$ (and $\ZeroOrderVasyDown$ by a constant as the curvature is bounded on $\Xone$) so we can find $C$ independent of $\lambda$ such that
\[
|\LinnprodSS{\QVasyDOWNL \tilde u}{\tilde u} |\ge  C^{-1} \|\n \tilde u\|_\s^2 + (\lambda^2 - C)\|\tilde u\|_\s^2
\]
and taking $\lambda\gg \sqrt{C}$ shows $\tilde u=0$ on $\{\rho>0\}$. By smoothness, $u$ vanishes on $\{\mu\ge0\}$ (and so to do all its derivates on $Y$). Standard hyperbolic estimates give the desired result $u=0$ if we can show a type of unique continuation result that $u=0$ on $\{\mu>-\ve\}$. 

To this end we work on $U^2$ and consider $\PVasyUp$ written in the following form
\[
\PVasyDOWNL = -\mu\partial_\mu^2 + \Delta_h + \BVasyDOWNL
\]
for $\BVasyDOWNL=-4\lambda\partial_\mu + \LowerOrderTermVasyDOWNL \in \Diff^1(U^2;\End\E)$. Let $\innprodHT{\cdot}{\cdot}$ on $\T^*Y\otimes \E$ denote the coupling of the metrics $h$ on $\T^*Y$ with $\innprodTT{\cdot}{\cdot}$ on $\E$. For ease of presentation, we will assume throughout this demonstration that all objects are real-valued.  
Consider $u,v\in C_c^\infty(U^2,\E)$ (and we may assume $\supp \, u \subset(-1,0]\times Y$) then we have the following formula
\[
\innprodHT{\nY u}{ \nY v} = \innprodTT{\Delta_h u}{v} + \divterm
\]
where $\divterm$ denotes any term which is of divergence nature on $Y$, hence vanishes upon integrating over $Y$ (using $d\vol_h$). Indeed such an equation is obtained by considering $f\in C^\infty(Y)$ and calculating, at some value $\mu$,
\begin{align*}
\int_{Y} \innprodHT{\nY u}{\nY v} f d\vol_h
& = \int_Y \innprodHT{\nY u}{ \nY ( f v)} - \innprodHT{\nY u}{\nY f\otimes v} \,d\vol_h \\
& = \int_Y (\innprodTT{\Delta_h u}{v} + \divterm)f \, d\vol_h
\end{align*}
where the second term was dealt with in the following way:
\begin{align*}
\int_Y \innprodHT{ \nY u }{ \nY  f \otimes v }\, d\vol_h 
&= \int_Y \sum_i \innprodTT{ \nY_{e_i} u }{ v }  \trH (e^i \otimes \nY  f) \, d\vol_h  \\
&= \int_Y \nY^* ( \sum_i  \innprodTT{\nY_{e_i} u}{v} e^i ) f  \, d\vol_h.
\end{align*}
With this formula established we define, for given $u$,
\[
\HVasyDown (\mu) = |\mu| \innprodTT{\partial_\mu u}{\partial_\mu u} + \innprodHT{\nY u}{ \nY  u} + \innprodTT{u}{u}
\]
and on $\{\mu<0\}$ (using $v=\partial_\mu u$ in the previously established formula)
\[
-\partial_\mu\HVasyDown  = -2 \innprodTT{ \PVasyDown u }{ \partial_\mu u} + \innprodTT{ (2\BVasyDOWNL-\partial_\mu)u }{\partial_\mu u} + \divterm - \tilde \HVasyDown.
\]
where $\tilde\HVasyDown$ has the same structure as $\HVasyDown$ but with appearances of $h$ (used to construct the various inner products) replaced by its Lie derivative, $\mathcal{L}_{\partial_\mu} h$. Recall that $\supp\, u\subset (-1,0]\times Y$ and $u$ is smooth, hence $\partial_\mu^N u=0$ at $\{\mu=0\}$ for all $N$. Continuing to work on $\{\mu<0\}$,
\begin{align*}
-\partial_\mu ( |\mu|^{-N}\HVasyDown) &+ |\mu|^{-N}\divterm \\
&=
-N|\mu|^{-N-1}\HVasyDown - 2|\mu|^{-N} \Re \innprodTT{\PVasyDOWNL u}{\partial_\mu u} + |\mu|^{-N}  \innprodTT{(2\BVasyDOWNL-\partial_\lambda)u}{\partial_\mu u} - |\mu|^{-N} \tilde \HVasyDown.
\end{align*}
Now suppose that $u\in \ker\PVasyDOWNL$. Fix $\delta>0$ small and let $0<\ve<\delta$. We take the previous display and insert it into the operator $\int_{-\delta}^{-\ve}\int_Y\dots \,d\mu\, d\vol_h$. The first term on the left hand side of the previous display is treated with the fundamental theorem of calculus, the second term vanishes due to the appearance of $\int_Y \divterm \,d\vol_h$. We claim the right hand side is negative for large $N$. Indeed the second term vanishes as $u$ is assumed in the kernel of $\PVasyDOWNL$. Considering the third term, $\innprodTT{(2\BVasyDOWNL-\partial_\lambda)u}{\partial_\mu u}$ is quadratic in $u$, $\nY u$, and $\partial_\mu u$ hence for $N$ large enough, it may be bounded by $N|\mu|^{-1}\HVasyDown$, thus the third term's potential positivity may be absorbed by the negativity of the first term. The fourth term may be treated in a similar manner upon consideration of the Taylor expansion of $h$ at $Y$. We obtain
\[
\delta^{-N} \int_Y \HVasyDown(-\delta) \, d\vol_h \le \ve^{-N} \int_Y \HVasyDown(-\ve) \, d\vol_h.
\]
As $u$ is smooth and vanishes to all orders at $\mu=0$, we may bound $ \int_Y \HVasyDown(-\ve) \, d\vol_h$ by $C |\mu|^K$ on $[-\ve,0]$ for arbitrarily large $K$. We can obtain a similar bound for $ \int_Y \HVasyDown(-\ve) \, d\vol_h$. In particular, for $K>N$. This produces
\[
\delta^{-N} \int_Y \HVasyDown(-\delta) \, d\vol_h \le  C \ve^{-N+K}
\]
and letting $\ve\to 0^+$ shows $\int_Y \HVasyDown(-\delta)\,\d\vol_h=0$ hence $\HVasyDown(-\delta)=0$. Doing this for all $\delta$ less than the original $\delta$ gives $\HVasyDown=0$ near $0$. Hence $\partial_\mu u$ and $\n^Y u$ vanish and $u=0$ near $0$. This suffices to conclude the proof.
\end{proof}
\begin{lemma}\label{lem:Ptrivialcokernel}
For $\PVasyDOWNL^*$ with $\lambda\in\R$ acting on $\dot H^{1-s}(\overline\Xtwo;\E)$, the kernel of $\PVasyDOWNL^*$ is trivial for $\lambda\gg1$.
\end{lemma}
\begin{proof}
Take $\lambda$ satisfying the threshold condition and consider $v\in\ker\PVasyDOWNL^*$. Hyperbolicity, as used in Lemma~\ref{lem:Pfinitecokernel}, implies $v=0$ on $\{\mu\le0\}$, and that $v$ is smooth on $\Xone$ due to ellipticity. 
The strategy given in Lemma~\ref{lem:Pfinitecokernel} implies $v\in \dot H^{\tilde s}(\overline\Xtwo;\E)$ for all $\tilde s < \lambda + \tfrac 12$ which with $\lambda\gg n$ implies $v$ is continuous. 
By the same logic, again by taking $\lambda$ sufficiently large, we may assume $v$ is regular enough to conclude ${\partial_\mu^N v}_{|Y}=0$ for $N\le \tfrac 12 \lambda$. 
Equivalently,
$v_{|\Xone} \in \rho^{2N} C^\infty_\even(\overline\Xone;\E)$.
Meanwhile, direct calculations on $C^\infty(\Xone;\E)$ give
\begin{align*}
\rho^{N}\n^*\n \rho^{-N} &= \n^*\n - N^2 - N(\Delta \log \rho) +2N\n_{\rdr}, \\
\rho^{N}\d \rho^{-N} &= d - N \drr\sp, \\
\rho^{N}  \div  \rho^{-N} &= \div + N \hookprod{\drr}{}
\end{align*}
where $\Delta \log \rho = n - (\tfrac 12 \sum_{ij} h^{ij} \rdr h_{ij})\in n-\rho^2 C^\infty_\even(\overline\Xone;\E)$. Also for $\tilde u \in C^\infty_c(\Xone;\E)$ we have
\[
|\LinnprodSS{ 2N \n_{\rdr} \tilde u }{\tilde u }|
=
|N \int_\Xone \| u \|_\s^2 \partial_\rho \left(\tfrac{d\rho \, d\vol_h}{\rho^{n}}\right)|
\le
CN \| u \|_\s^2.
\]
So consider the difference operator $(\QVasyDOWNL - N^2+2N\n_{\rdr}) - \rho^{N}\QVasyDOWNL \rho^{-N}$ acting on $\tilde u\in C^\infty_c(\Xone;\E)$. All terms are of order $N$ and of differential order 0. Similar to the previous proof (and using the preceding remark in order to treat the term involving $N\n_{\rdr}$) we may obtain
\[
|\LinnprodSS{\rho^{N} \QVasyDOWNL \rho^{-N} \tilde u }{ \tilde u} | \ge C^{-1} \|\n \tilde u\|_\s^2 + (\lambda^2 - N^2 - CN)\| \tilde u\|_\s^2
\]
and provided $N\gg C$, the final term in the preceding display may be written with coefficient $\lambda^2-2N^2$. Set $N=\lfloor \frac 12 \lambda \rfloor$ with $\lambda\gg 2C$. So that
\[
|\LinnprodSS{\rho^{N} \QVasyDOWNL \rho^{-N} \tilde u }{ \tilde u} | \ge C^{-1} \|\n \tilde u\|_\s^2 + \tfrac12 \lambda^2 \| \tilde u\|_\s^2.
\]
Considering the Hilbert space $\{ w\in \LsectionsS(\Xone;\E) : B(w,w)<\infty \}$ with $B(w,w)=\| \rho^N \QVasyDOWNL \rho^{-N} w\|_\s^2<\infty$, the previous inequality shows that $w\mapsto \LinnprodSS{w}{\tilde f}$ is a linear functional for $\tilde f\in \LsectionsS(\Xone;\E)$ so by the Riesz representation theorem, there exists $\tilde u\in \LsectionsS(\Xone;\E)$ with $\LinnprodSS{\rho^N\QVasyDOWNL \rho^{-N} w}{\tilde u} = \LinnprodSS{w}{\tilde f}$ for all $w$. To show $v$ vanishes on $\Xone$, it suffices to show $\LinnprodTT{f}{v}=0$ for all $f\in C^\infty_c(\Xone;\E)$. Let $f\in C^\infty_c(\Xone;\E)$ and
\[
\tilde f = \rho^{\lambda+\frac n2 - m + 2} J^{-1} \rho^{-N} f \in C^\infty_c(\Xone;\E)
\]
Then the preceding argument gives $\tilde u\in \LsectionsS(\Xone;\E)$ such that $\rho^{-N} \QVasyDOWNL \rho^N \tilde u = \tilde f$ hence $\PVasyDOWNL u = f$ where
\[
u = J \rho^{-\lambda - \frac n2 + m} \rho^{N} \tilde u \in \rho^{-\frac 12 \lambda + 1} \LsectionsT(\Xone;\E)
\]
(the inclusion is a consequence of Lemma~\ref{lem:innprod:changeofscales}). This gives $u$ enough regularity to perform the following pairing which provides the desired result
\[
\LinnprodTT{f}{v} = \LinnprodTT{\PVasyDOWNL u}{v} = \LinnprodTT{u}{\PVasyDOWNL^* v} = \LinnprodTT{u}{0} = 0. \qedhere
\] 
\end{proof}

%%%%% %%%%% %%%%% %%%%% %%%%% %%%%% %%%%% %%%%% %%%%% %%%%% %%%%% %%%%%
%%%%% %%%%% %%%%% %%%%% %%%%% %%%%% %%%%% %%%%% %%%%% %%%%% %%%%% %%%%%
%%%%% %%%%% %%%%% %%%%% %%%%% %%%%% %%%%% %%%%% %%%%% %%%%% %%%%% %%%%%

\section{Proofs of Theorems}\label{sec:proofs:main}

\subsection{Proof of Theorem~\ref{thm:QQ}}
Proposition~\ref{prop:PPandQQ:changeofscales} gives
\[
\QVasyDOWNL = J^{-1} \rho^{\lambda + \frac n2 - m +2} \PVasyDOWNL \rho^{-\lambda  - \frac n2 + m} J
\]
By Propositions~\ref{prop:fredholm} and~\ref{prop:inverse:existence}, there is a meromorphic family $\PVasyDown^{-1}$ on $\C$ mapping $\dot C^\infty(\Xone;\E)$ to $C^\infty(\Xtwo;\E)$. Hence an extension of $\QVasyDown^{-1}$ from $\Re\lambda\gg1$ to all of $\C$ as a meromorphic family is given by
\[
\QVasyDOWNL^{-1} = J^{-1} \rlnm r_\Xone \PVasyDOWNL^{-1} \rho^{-\lambda - \frac n2 + m - 2} J
\]
where $r_\Xone$ is the restriction of sections above $\Xtwo$ to sections above $\Xone$. The previous display implies
\[
\QVasyDOWNL^{-1} : \dot C^\infty ( \Xone ; \E )\to \rlnm J^{-1} C^{\infty}_\even ( \overline\Xone ; \E )
\]
and for $f\in \dot C^\infty(\Xone;\E)$, we may write near $\partial\overline\Xone$
\[
{\QVasyDOWNL^{-1} f}_{|U} = \mu^{\frac\lambda2 + \frac n4-\frac m2}J^{-1} \sum_{k=0}^m \sum_{\ell=0}^k (d\mu)^{k-\ell}\sp \tilde u^{(\ell)},
\qquad
\tilde u^{(\ell)} \in C^\infty_\even([0,1)\times Y ; \Sym^\ell \T^*Y)
\]
The proof of Lemma~\ref{lem:innprod:changeofscales} shows that the part of $J$ (or $J^{-1}$) which sends $\Ek$ to $\E^{(k+p)}$ for $0\le p \le m-k$ is, up to a constant, $(\frac{d\mu}{\mu})^p$. Therefore,
\[
{\QVasyDOWNL^{-1} f}_{|U} \in\mu^{\frac\lambda2 + \frac n4-\frac m2}\bigoplus_{k=0}^m \bigoplus_{p=0}^{m-k} (\tfrac{d\mu}{\mu})^p\sp \bigoplus_{\ell=0}^{k} (d\mu)^{k-\ell} \sp C^{\infty}_\even ( [0,1)\times Y ; \Sym^\ell \T^*Y )
\]
hence on $\Xone$,
\[
\QVasyDOWNL^{-1} f \in \rlnm \bigoplus_{k=0}^m \bigoplus_{p=0}^{m-k} \rho^{-2p} C^\infty_\even(\overline\Xone;\E^{(k+p)})
\]
which is contained in $\rlnm \bigoplus_{k=0}^m \rho^{-2k} C^\infty_\even(\overline\Xone;\Ek)$.
\begin{remark}\label{rem:asymptotics}
Suppose that, for $f\in \dot C^\infty(\Xone;\E)$, it were possible to write in the preceeding proof that near $\partial\overline\Xone$
\[
{\QVasyDOWNL^{-1} f}_{|U} = \rlnm J^{-1} \tilde u^{(m)},
\qquad
\tilde u^{(m)} \in C^\infty_\even(\overline U ; \E^{(m)})
\]
then as $J^{-1}$ acts as the identity upon restriction to $\E^{(m)}$, we would obtain
\[
\QVasyDOWNL^{-1} f \in \rlnm C^\infty_\even(\overline\Xone;\E^{(m)})
\]
This will be useful for the asymptotics given in Theorems~\ref{thm:main:2} and~\ref{thm:main:m}.
\end{remark}

\subsection{Proof of Theorem~\ref{thm:QQ:tracefree}}

The meromorphic inverse of $\QVasyDOWNL$ is precisely that given in the preceding proof
\[
\QVasyDOWNL^{-1} = J^{-1} \rlnm r_\Xone \PVasyDOWNL^{-1} \rho^{-\lambda - \frac n2 + m - 2} J.
\]
All we must check is, given $f\in \dot C^\infty(\Xone;\E)\cap\ker(\traceSETA\circ\pistarS)$, that the resulting section $u=\QVasyDOWNL^{-1}f$ is indeed trace-free with respect to the ambient trace operator. To this end, we first lift the equation $\QVasyDOWNL u=f$ to an equation on $\Mone$ involving $\QVasyUp$ giving
\[
\s^{\lambda} \QVasyUp \s^{-\lambda} \left(
													\pistarS u
										\right)
= \pistarS f.
\]
We apply $\traceSETA$ to obtain an equation on $\F^{(m-2)}$. Using the hypothesis $\traceSETA\pistarS f=0$ and Lemma~\ref{lem:Qcommutestrace} to commute $\s^{-2}\traceSETA$ with $\QVasyUp$ gives
\[
\s^2\s^{\lambda}\QVasyUp\s^{-\lambda}\s^{-2}
 \traceSETA\left( \pistarS u \right) = 0.
\]
Freezing this differential equation at $\s=0$ with $\pi_{\s=0}$ to obtain the indicial family of $\QVasyUp$ provides the equation
\[
\IF_\s(\QVasyUp,\lambda+2) \pi_{\s=0}  \traceSETA\left( \pistarS u \right) = 0.
\]
Section~\ref{sec:analysis} ensures that for $\Re\lambda\gg1$, this operator has trival kernel hence
\[
\pi_{\s=0}  \traceSETA\left( \pistarS u \right) = 0
\]
and $u\in\ker(\traceSETA\circ\pistarS)$ as required.

\subsection{Proof of Theorems~\ref{thm:main:2} and~\ref{thm:main:m}}

We are finally in a position to consider the original problem of proving Theorems~\ref{thm:main:2} and~\ref{thm:main:m}. Let 
\[
f\in \dot C^\infty(\Xone; \E^{(m)})\cap\ker\trace\cap\ker \div 
\]
and define, using Theorem~\ref{thm:QQ},
\[
u
=
\sum_{k=0}^m u^{(k)}
= 
\QVasyDOWNL^{-1} f 
%= J^{-1} \rho^{-\lambda - m + \frac n2} r_\Xone \PVasyDOWNL^{-1} \rho^{\lambda + m - \frac n2 - 2} J f
,\qquad
u^{(k)} \in  \rho^{\lambda + \frac n2 - m  - 2k} C^\infty_\even (\overline\Xone; \Ek).
\]
Note that the growth near $\partial \overline\Xone$ of $u^{(k)}$ and $ \div  u^{(k)}$ may be controlled by the size of $\Re\lambda$ hence for $\Re\lambda\gg1$ we may assume that they are sections of $\LsectionsS(\Xone;\Ek)$ and $\LsectionsS(\Xone;\E^{(k-1)})$ respectively. We claim, for $\Re\lambda\gg1$ and $|\Im \lambda| \ll 1$, that 
\[
u=u^{(m)} \in \rlnm C^\infty_\even (\overline\Xone; \Ek) \cap \ker \trace \cap \ker  \div 
\]
at which point the equation $\QVasyDOWNL u = f$ decouples giving
\[
(\Delta+\lambda^2-c_m)u = f
\]
and by uniqueness of the $\Lsections$ inverse of the Laplacian, we have the formula, for $\Re\lambda\gg1$ and $|\Im \lambda| \ll 1$,
\[
(\Delta +\lambda^2 - c_m)^{-1} = J^{-1} \rlnm r_\Xone \PVasyDOWNL^{-1} \rho^{-\lambda-\frac n2 +m-2} J.
\]
with the right hand side giving the meromorphic extension of the resolvent announced in the theorems.

To this end take $\Re\lambda\gg1$ and $|\Im \lambda| \ll 1$. By Theorem~\ref{thm:QQ:tracefree}, we deduce $u$ is trace-free with respect to the ambient trace operator thus $\QVasyDOWNL$ takes the form detailed in Proposition~\ref{prop:decompQQ:k:mink:tracefree}.
We begin by remarking, that while working on $\LsectionsS(\Xone;\Ek)$ if $\Resl^{(k)}$ is any operator of the form $(\Delta+\lambda^2+O(1))^{-1}$ (which has order $O(|\lambda|^{-2})$, then the operator $\d\Resl^{(k)} \div $ has norm of order $O(1)$. We define $\Resl^{(0)}=(\Delta+\lambda^2-c_0')^{-1}$ and for $0<k<m$,
\[
\Resl^{(k)} = \left(\Delta+\lambda^2 -c_k' + 4(m-k+1) \d\Resl^{(k-1)} \div \right)^{-1}.
\]
The component of $\QVasyDOWNL u = f$ in $\E^{(0)}$ reads,
\[
(\Delta+\lambda^2 - c_0')u^{(0)} = 2\sqrt m  \div  u^{(1)}
\]
hence $u^{(0)} = 2\sqrt{m} \Resl^{(0)}  \div  u^{(1)}$. The component of $\QVasyDOWNL u = f$ in $\E^{(1)}$ now reads,
\[
(\Delta+\lambda^2 - c_1' + 4m\d\Resl^{(0)} \div )u^{(1)} = 2\sqrt{m-1}  \div  u^{(2)}
\]
hence $u^{(1)} = 2\sqrt{m-1} \Resl^{(1)}  \div  u^{(2)}$. Continuing, we obtain on $\E^{(m)}$,
\[
(\Delta+\lambda^2 - c_m + 4 \d\Resl^{(m-1)} \div )u^{(m)} = f.
\]
Applying the divergence, we recall Lemma~\ref{lem:laplaciancommutestracediv}. For this, we must assume that if $m=2$ then $\Xone$ has parallel Ricci curvature, and if $m\ge 3$ then $\Xone$ is locally isomorphic to $\Hn$. We obtain,
\[
(\Delta+\lambda^2 - c_m + 4  \div \d\Resl^{(m-1)}) \div u^{(m)}=0.
\]
Again, $ \div \d\Resl^{(m-1)}$ has norm of order $O(1)$ so we may invert this equation and deduce that $\div  u^{(m)}=0$. This implies, for all $k<m$, 
\[
u^{(k)} = 2\sqrt{m-k} \Resl^{(k)}  \div  u^{k+1} = 0.
\]
Therefore $u=u^{(m)}$. By Remark~\ref{rem:asymptotics}, $u\in \rlnm C^\infty_\even (\overline\Xone; \E^{(m)})$. By Theorem~\ref{thm:QQ:tracefree}, $u\in\ker\trace$. And as previously mentioned $u\in\ker \div $. This completes the proof.

\section{Symmetric cotensors of rank 2}\label{sec:m=2}

This section details the results announced in Sections~\ref{sec:Lap.DAlem.Q} and~\ref{sec:proofs:main} for rank 2 symmetric cotensors. In this low rank, writing the action of the d'Alembertian, or its conjugation $\QVasyUp$, on $\F=\Sym^2\T^*\Mone$ is tractable.

\subsection{The operator $\QVasyUp$ for 2-cotensors}\label{subsec:Minkowski:2}

Using the decomposition given by the Minkowksi scale, we write
\[
u = 
\threerow{1}{\dss\sp}{\tfrac{1}{\sqrt2}(\dss)^2}
\threevector{u^{(2)}}{u^{(1)}}{u^{(0)}},
\qquad
u\in C^\infty(\Mone; \F), 
u^{(k)}\in C^\infty(\Mone;\Ek)
\]
The change of basis matrix $J$ takes the form
\begin{align*}
J & = \threematrix	{1}{\drr\sp}{\tfrac{1}{\sqrt2}(\drr)^2}
				{0}{1}{\sqrt2\drr\sp}
				{0}{0}{1}.
\end{align*}
Propositions~\ref{prop:decompQ:k:mink} and~\ref{prop:decompQ:k:mink:tracefree} become
\begin{proposition}\label{prop:decompQ:2:mink}
For $u\in C^\infty(\Mone; \F)$ decomposed relative to the Minkowski scale \eqref{eqn:decomp:mink}, the conjugated d'Alembertian $\QVasyUp$ is given by
\[
\QVasyUp u = 	\threerow{1}{\dss\sp}{\tfrac{1}{\sqrt2}(\dss)^2}
		\threematrix	{\Delta +(\sds)^2 - c_2-\adjtrace\trace}		{2\d}							{-\sqrt2\adjtrace}
							{-2 \div}			{\Delta + (\sds)^2 - c_1}			{2\sqrt2\d}
							{-\sqrt2\trace}		{-2\sqrt2 \div}				{\Delta + (\sds)^2 - c_0}
	\threevector{u^{(2)}}{u^{(1)}}{u^{(0)}}
\]
with constants 
\[
c_2 = \tfrac{1}{4}n(n-8), \qquad c_1=\tfrac{1}{4}(n^2+16), \qquad c_0=\tfrac{1}{4}(n^2+8n+8).
\]
If, furthermore, $u$ is trace-free with respect to the trace operator $\traceSETA$, then $\trace u^{(2)}=-\sqrt 2u^{(0)}$, and
\[
\QVasyUp   u = 	\threerow{1}{\dss\sp}{\tfrac{1}{\sqrt2}(\dss)^2}
\threematrix	{\Delta +(\sds)^2 - c_2'}		{2\d}					{0}
							{-2\div}			{\Delta +(\sds)^2 - c_1'}		{2\sqrt2\d}
							{0}					{-2 \sqrt2\div}				{\Delta +(\sds)^2 - c_0'}
	\threevector{u^{(2)}}{u^{(1)}}{u^{(0)}}
\]
with modified constants
\[
c_2' = c_2,
\qquad
c_1'=c_1,
\qquad
c_0'=\tfrac{1}{4}(n^2+8n).
\]
\end{proposition}

\subsection{The indicial family of $\QVasyUp$ for 2-cotensors}

Propositions~\ref{prop:decompQQ:k:mink} and~\ref{prop:decompQQ:k:mink:tracefree} become

\begin{proposition}\label{prop:decompQQ:2:mink}
For $u=\sum_{k=0}^2 u^{(k)} \in C^\infty(\Xone; \E)$ the operator $\QVasyDown$ is given by
\[
\QVasyDOWNL  u = \threematrix	{\Delta +\lambda^2 - c_2-\adjtrace\trace}		{2\d}							{-\sqrt2\adjtrace}
							{-2 \div}			{\Delta + \lambda^2 - c_1}			{2\sqrt2\d}
							{-\sqrt2\trace}		{-2\sqrt2 \div}				{\Delta + \lambda^2 - c_0}
	\threevector{u^{(2)}}{u^{(1)}}{u^{(0)}}
\]
and if, furthermore, $u\in\ker(\traceSETA\circ\pistarS)$ then
\[
\QVasyDOWNL u = 	\threematrix	{\Delta +\lambda^2 - c_2'}		{2\d}					{0}
							{-2\div}			{\Delta +\lambda^2 - c_1'}		{2\sqrt2\d}
							{0}					{-2 \sqrt2\div}				{\Delta +\lambda^2 - c_0'}
	\threevector{u^{(2)}}{u^{(1)}}{u^{(0)}}
\]
with previously announced constants.
\end{proposition}

\subsection{Illustration of proof for 2-cotensors}

Let $f\in \dot C^\infty(\Xone; \E^{(2)})\cap\ker\trace\cap\ker\div$ and define 
\[
\threevector{u^{(2)}}{u^{(1)}}{u^{(0)}}
= J^{-1} \rho^{\lambda + \frac n2 - 2} r_\Xone\PVasyDown^{-1} \rho^{-\lambda-\frac n2} J
\threevector{f}{0}{0}
\]
Take $\Re\lambda\gg1$ and $|\Im \lambda| \ll 1$. By Theorem~\ref{thm:QQ}
\[
u^{(k)} \in  \rho^{\lambda + \frac n2 - 2  - 2k} C^\infty_\even (\overline\Xone; \Ek)
\]
and by Proposition~\ref{prop:PPandQQ:changeofscales}, $\QVasyDOWNL u =f$. Theorem~\ref{thm:QQ:tracefree} forces
\[
\traceSETA \left( u^{(2)} + \dss \sp u^{(1)} + \tfrac1{\sqrt2} (\dss)^2.u^{(0)}
			\right) = 0
\]
hence $\trace u^{(2)}=-\sqrt 2u^{(0)}$. And $\QVasyDOWNL u =f$ reads explicitly
\[
\threematrix	{\Delta +\lambda^2 - c_2}		{2\d}					{0}
							{-2\div}			{\Delta +\lambda^2 - c_1}		{2\sqrt2\d}
							{0}					{-2 \sqrt2\div}				{\Delta +\lambda^2 - c_0'}
	\threevector{u^{(2)}}{u^{(1)}}{u^{(0)}}
	=
	\threevector{f}{0}{0}
\]
Introducing the resolvents $\Resl^{(0)}$ and $\Resl^{(1)}$ provides
\[
\threematrix	{\Delta +\lambda^2 - c_2 + 4\d\Resl^{(1)}\div}		{0}					{0}
							{-2\div}			{\Delta +\lambda^2 - c_1+ 8\d\Resl^{(0)}\div}		{0}
							{0}					{-2 \sqrt2\div}				{\Delta +\lambda^2 - c_0'}
	\threevector{u^{(2)}}{u^{(1)}}{u^{(0)}}
	=
	\threevector{f}{0}{0}
\]
and applying $\div$ assuming that $\Xone$ is Einstein provides the homogeneous equation
\[
\threematrix	{\Delta +\lambda^2 - c_2 + 4\div\d\Resl^{(1)}}		{0}					{0}
							{-2\div}			{\Delta +\lambda^2 - c_1+ 8\div\d\Resl^{(0)}}		{0}
							{0}					{-2 \sqrt2\div}				{\Delta +\lambda^2 - c_0'}
	\threevector{\div u^{(2)}}{\div  u^{(1)}}{\div u^{(0)}}
	=
	\threevector{0}{0}{0}.
\]
The lower triangular nature of this system implies $\div u^{(k)}=0$ for all $k$. Hence the system $\QVasyDOWNL u = f$ collapses. So $u^{(0)}$ and $u^{(1)}$ vanish and by Remark~\ref{rem:asymptotics},
\[
u=u^{(2)} \in  \rho^{\lambda + \frac n2 - 2} C^\infty_\even (\overline\Xone; \Ek)
\]
giving $(\Delta +\lambda^2 - c_2) u = f$.

%%%%% %%%%% %%%%% %%%%% %%%%% %%%%% %%%%% %%%%% %%%%% %%%%% %%%%% %%%%%
%%%%% %%%%% %%%%% %%%%% %%%%% %%%%% %%%%% %%%%% %%%%% %%%%% %%%%% %%%%%
%%%%% %%%%% %%%%% %%%%% %%%%% %%%%% %%%%% %%%%% %%%%% %%%%% %%%%% %%%%%

\section{High Energy Estimates via Semiclassical Analysis}\label{sec:highenergyestimates}

This article shows the meromorphic continuation of the resolvent of the Laplacian on symmetric tensors using microlocal techniques. This direction means one does not talk about introducing complex absorbers but rather studies the problem on a manifold with boundary. If one were to follow more closely the track established by Vasy, one obtains semiclassical estimates. We state these estimates.

On $\Xone$, whose smooth structure at infinity is the even structure given by $\mu$ rather than $\rho$, we have the semiclassical spaces $H^s_{|\lambda|^{-1}}(\Xone;\E)$.
\begin{theorem}
Suppose that $\Xone$ is an even asymptotically hyperbolic manifold which is non-trapping. Then the meromorphic continuation, written $\QVasyDOWNL^{-1}$ of the inverse of $\QVasyDOWNL$ initially acting on $\LsectionsS( \Xone ; \E )$
has non-trapping estimates holding in every strip $|\Re\lambda|<C, |\Im\lambda|\gg0$: for $s > \frac 12 + C$
\[
\| \rho^{-\lambda -\frac n2 + m} \QVasyDOWNL^{-1} f \|_{H^s_{|\lambda|^{-1}}(\Xone;\E)}
\le
C |\lambda|^{-1} \| \rho^{-\lambda-\frac n2 + m-2} f\|_{H^{s-1}_{|\lambda|^{-1}}(\Xone;\E)}.
\]
If $\Xone$ is furthermore Einstein, then restricting to symmetric 2-cotensors, the meromorphic continuation $\Resl$ of the inverse of
\[
\Delta - \frac{n(n-8)}{4} +\lambda^2
\]
initially acting on $\Lsections( \Xone ; \E^{(2)} )\cap\ker\trace\cap\ker\div$
has non-trapping estimates holding in every strip $|\Re\lambda|<C, |\Im\lambda|\gg0$: for $s > \frac 12 + C$
\[
\| \rho^{-\lambda -\frac n2 + 2} \Resl f \|_{H^s_{|\lambda|^{-1}}(\Xone;\E^{(2)})}
\le
C |\lambda|^{-1} \| \rho^{-\lambda-\frac n2} f\|_{H^{s-1}_{|\lambda|^{-1}}(\Xone;\E^{(2)})}.
\]
\end{theorem}

%%%%% %%%%% %%%%% %%%%% %%%%% %%%%% %%%%% %%%%% %%%%% %%%%% %%%%% %%%%%
%%%%% %%%%% %%%%% %%%%% %%%%% %%%%% %%%%% %%%%% %%%%% %%%%% %%%%% %%%%%
%%%%% %%%%% %%%%% %%%%% %%%%% %%%%% %%%%% %%%%% %%%%% %%%%% %%%%% %%%%%

%%%%% %%%%% %%%%% %%%%% %%%%% %%%%% %%%%% %%%%% %%%%% %%%%% %%%%% %%%%%
%%%%% %%%%% %%%%% %%%%% %%%%% %%%%% %%%%% %%%%% %%%%% %%%%% %%%%% %%%%%
%%%%% %%%%% %%%%% %%%%% %%%%% %%%%% %%%%% %%%%% %%%%% %%%%% %%%%% %%%%%

%%%%% %%%%% %%%%% %%%%% %%%%% %%%%% %%%%% %%%%% %%%%% %%%%% %%%%% %%%%%

% arXiv bibliography macro
\def\arXiv#1{\href{http://arxiv.org/abs/#1}{arXiv:#1}}
\bibliographystyle{alpha}

%\input{bib}

%%%%% %%%%% %%%%% %%%%% %%%%% %%%%% %%%%% %%%%% %%%%% %%%%% %%%%% %%%%%
%%%%% %%%%% %%%%% %%%%% %%%%% %%%%% %%%%% %%%%% %%%%% %%%%% %%%%% %%%%%
%%%%% %%%%% %%%%% %%%%% %%%%% %%%%% %%%%% %%%%% %%%%% %%%%% %%%%% %%%%%

\end{document}